\documentclass[a4paper]{amsart}
\usepackage{url}
\usepackage{amsmath,amssymb,amsfonts,amsthm,bbm}
\usepackage[shortlabels]{enumitem}
\usepackage{todonotes}
\usepackage{algorithmic}
\usepackage[pdftex]{hyperref}
\usepackage[T1]{fontenc}

\usepackage{graphicx}

\usepackage{color}

\renewcommand{\AA}{\mathbbm{A}}
\newcommand{\av}{\vec{a}}

\newcommand{\bv}{\vec{b}}
\newcommand{\cv}{\vec{c}}
\newcommand{\dv}{\vec{d}}
\newcommand{\C}{\mathcal{C}}
\newcommand{\CC}{\mathbbm{C}}
\newcommand{\D}{\mathcal{D}}

\newcommand{\E}{\mathcal{E}}
\newcommand{\ev}{\vec{e}}
\newcommand{\F}{\mathcal{F}}
\newcommand{\fv}{\vec{f}}

\renewcommand{\H}{\mathcal{H}}
\newcommand{\I}{\mathcal{I}}
\newcommand{\id}[1]{\mathrm{id}_{#1}}
\DeclareMathOperator{\init}{init}
\DeclareMathOperator{\im}{im}
\newcommand{\J}{\mathcal{J}}
\newcommand{\Jc}[1]{\J_{#1}}
\newcommand{\Jcp}[2]{\J_{#1}^{(#2)}}

\newcommand{\Jhc}[1]{\hat{\J}_{#1}}
\newcommand{\je}[1]{J_{#1}{\pi}}
\newcommand{\K}{\mathcal{K}}
\DeclareMathOperator{\ld}{ld}
\newcommand{\M}{\mathcal{M}}
\newcommand{\N}{\mathcal{N}}
\newcommand{\Nc}[2]{\N_{#2}[\Jc{#1}]}
\newcommand{\NN}{\mathbbm{N}}
\newcommand{\Nn}{\NN_0^n}
\renewcommand{\O}{\mathcal{O}}
\DeclareMathOperator{\ord}{ord}
\renewcommand{\P}{\mathcal{P}}
\DeclareMathOperator{\pp}{pp}
\newcommand{\R}{\mathcal{R}}
\newcommand{\Rc}[1]{\R_{#1}}

\DeclareMathOperator{\spt}{sep}
\DeclareMathOperator{\Sol}{Sol}

\newcommand{\U}{\mathcal{U}}
\newcommand{\uv}{\vec{u}}
\newcommand{\vv}{\vec{V}}
\newcommand{\JJ}{\vec{J}}
\newcommand{\V}{\mathcal{V}}
\renewcommand{\vec}[1]{\mathbf{#1}}

\newcommand{\X}{\mathcal{X}}
\newcommand{\xv}{\vec{x}}

\newtheorem{theorem}{Theorem}[section]
\newtheorem{proposition}[theorem]{Proposition}
\newtheorem{lemma}[theorem]{Lemma}
\newtheorem{corollary}[theorem]{Corollary}
\theoremstyle{definition}
\newtheorem{definition}[theorem]{Definition}
\newtheorem{example}[theorem]{Example}
\theoremstyle{remark}
\newtheorem{remark}[theorem]{Remark}
\newtheorem{algorithm}[theorem]{Algorithm}

\begin{document}

\title{Singularities of Algebraic Differential Equations}
\author[M. Lange-Hegermann]{Markus Lange-Hegermann}
\address{Technische Hochschule Ostwestfalen Lippe - University of
  Applied Sciences and Arts, 32657 Lemgo, Germany} 
\email{markus.lange.hegermann@th-owl.de}
\author[D. Robertz]{Daniel Robertz}
\address{School of Engineering, Computing and Mathematics,
  University of Plymouth, 2-5 Kirkby Place, Drake Circus, Plymouth PL4 8AA, 
  United Kingdom} 
\email{daniel.robertz@plymouth.ac.uk}
\author[W.M. Seiler]{Werner M. Seiler}
\address{Institut f\"ur Mathematik, Universit\"at Kassel, 34109
  Kassel, Germany}
\email{seiler@mathematik.uni-kassel.de}
\author[M. Sei\ss]{Matthias Sei\ss}
\address{Institut f\"ur Mathematik, Universit\"at Kassel, 34109
  Kassel, Germany}
\email{mseiss@mathematik.uni-kassel.de}
\thanks{The work of the last two authors was partially supported by the
bilateral project ANR-17-CE40-0036 and DFG-391322026 SYMBIONT}


  
\begin{abstract}
  There exists a well established differential topological theory of
  singularities of ordinary differential equations.  It has mainly studied
  scalar equations of low order.  We propose an extension of the key
  concepts to arbitrary systems of ordinary or partial differential
  equations.  Furthermore, we show how a combination of this geometric
  theory with (differential) algebraic tools allows us to make parts of the
  theory algorithmic.  Our three main results are firstly a proof that even
  in the case of partial differential equations regular points are generic.
  Secondly, we present an algorithm for the effective detection of all
  singularities at a given order or, more precisely, for the determination
  of a regularity decomposition.  Finally, we give a rigorous definition of
  a regular differential equation, a notoriously difficult notion
  ubiquitous in the geometric theory of differential equations, and show
  that our algorithm extracts from each prime component a regular
  differential equation.  Our main tools are on the one hand the algebraic
  resp.\ differential Thomas decomposition and on the other hand the
  Vessiot theory of differential equations.
\end{abstract}

\maketitle

\section{Introduction}

Many different forms of singular behaviour appear in the context of
differential equations and many different views have been developed for
them.  Most of them are related to singularities of individual solutions of
a given differential equation like blow-ups or shocks, i.\,e.\ either a
solution component or some derivative of it becomes infinite.  By contrast,
we will be concerned with singularities of the differential equation
itself.  Using the geometric theory of differential equations
\cite{sau:jet,wms:invol}, (systems of) differential equations are
identified with subsets of suitable jet bundles and singularities are
special points on these subsets.

Within differential topology, singularities of smooth maps between
manifolds \cite{agv:sing1,gg:stable} have been much studied.  The
\emph{geometric singularities} of differential equations, which is our main
topic here, may be viewed as a special case (overviews over some basic
results can be found in \cite{via:geoode} or \cite{aor:poincare}).  The
main emphasis in the literature has been on the classification of
singularities (see e.\,g.\ \cite{diis:gensing}) and on the construction of
local normal forms for them.  Of course, such questions can be reasonably
treated only in sufficiently small dimensions and hence most works consider
only scalar ordinary differential equations of first or second order.  With
similar techniques, singularities of solutions of partial differential
equations have been studied e.\,g.\ in \cite{ag:whitney,vvl:singsol}, but
as already mentioned this represents a different problem.

By contrast, we are concerned with the effective treatment of general
systems of differential equations, i.\,e.\ also of under- or overdetermined
systems of ordinary or partial differential equations.  For this purpose,
we extend the needed concepts from differential topology to systems which
are not of finite type and we combine them with (differential) algebraic
algorithms to make them effective.  Such a combination of geometric and
algebraic approaches to singularities appeared already in the work of
Hubert \cite{eh:degen} on scalar first-order ordinary differential
equations.  However, we cover much more general situations than she did; in
particular, we admit systems, equations of arbitrary order and partial
differential equations.

We concentrate in this work on the definition and the algorithmic detection
of singularities of general differential systems.  The analysis of the
local solution behaviour around a singularity represents a much harder
problem that probably cannot be solved at the same level of generality or
effectivity.  The algebraic techniques employed by us require that we work
over the complex numbers and that we restrict to differential equations
with polynomial nonlinearities.  From the point of view of applications,
the latter restriction is not very serious, as most differential systems
arising in applied sciences are polynomial.

Studying fully nonlinear or implicit systems is not at all straightforward
and we need to address several challenges.  For systems of differential
equations, the corresponding subsets of jet bundles are no longer
hypersurfaces leading to a much more complicated relation between the given
differential system and the surfaces defined by it.  As a further
complication, general systems of differential equations may hide
integrability conditions, which must be exhibited explicitly before
statements about the existence and uniqueness of solutions can be made.
These facts make case distinctions (which are related to the appearance of
singularities) unavoidable.  Furthermore, in the case of partial
differential equations the completion may require to move to higher-order
jet bundles, so that a priori it is not even clear at what order any
further analysis should be performed.

Our approach proceeds in two steps: a differential one and an algebraic
one.  In the first step, we use the \emph{differential Thomas
  decomposition} \cite{th:ds, th:sr} (see \cite{bglr:thomasalg,
  bl:thomasimpl, vpg:decomp, glhr:tdds, dr:habil} for modern treatments) to
split the input system into a finite set of so-called \emph{simple}
differential systems.  Besides the splitting, the differential step also
takes care of the just mentioned problem of hidden integrability
conditions, as it includes a completion procedure.  Each of the arising
simple differential systems is then analysed separately.  This
decomposition also addresses singular integrals, which are automatically
isolated into separate simple systems, whereas the general integral
corresponds to other systems.  However, we do not claim to detect whether a
system corresponds to singular integrals, a problem which seems hard and is
closely connected to the so-called Ritt problem \cite[\textsection
IV.9]{kol:diffalg}.  An alternative to the Thomas decomposition is the
Rosenfeld-Gr\"obner algorithm \cite{blop:09}; the splittings it performs,
however, do not in general result in decompositions of the solution set
into pairwise disjoint subsets.  An elimination method for differential
algebra based on splittings analogous to Thomas' ones was developed by
Seidenberg \cite{seidenberg:elim}.

For the algebraic step, we must first choose a suitably high order in which
we want to analyse the simple differential system.  We associate with the
differential system a polynomial radical ideal in the coordinate ring of
the jet bundle of the chosen order and introduce this way algebraic jet
sets as a geometric model of the differential system
(Definition~\ref{def:ade}).  Over such sets, we study their Vessiot cones
which are fundamental for defining geometric singularities.  Using the
\emph{algebraic Thomas decomposition}, we partition algebraic jet sets with
respect to the behaviour of the Vessiot cones and show that such a
decomposition is equivalent to the identification of all geometric
singularities.  For finding algebraic singularities, we augment this
procedure with a suitable version of the Jacobian criterion from algebraic
geometry.

In the algebraic step, we must study more general situations than usually
considered in the differential topological approach to singularities.
Hence, we extend this approach in several directions.  We provide a more
general definition of geometric singularities that can also handle partial
differential equations (Definition~\ref{def:regsing}).  This requires a
considerably more involved definition taking into account a whole
neighbourhood of the studied point, whereas the classical definitions use
pointwise criteria.  In the case of systems, one can no longer expect that
singularities are isolated points, as it is traditionally done at least for
irregular singularities.  Therefore, we introduce the novel notion of a
\emph{regularity decomposition} of an algebraic jet set
(Definition~\ref{def:regdecomp}) as a partitioning into subsets on which
the relevant geometric structures (the Vessiot and symbol cones) show a
uniform behaviour.

Our first two main results concern these generalisations.
Theorem~\ref{thm:regdense} proves that the regular points form a Zariski
open and dense subset and thus justifies calling the other points singular.
In the situations traditionally considered in differential topology or
analysis, i.\,e.\ for differential equations of finite type, this statement
is fairly trivial.  As we also include equations which are not of finite
type, we must prove the existence of a smooth regular involutive
distribution of the right dimension on some neighbourhood of any regular
point which requires the application of advanced results from the geometric
theory of differential equations.  Our second main result concerns the
existence of regularity decompositions for arbitrary differential systems.
We provide here a constructive proof by providing an explicit algorithm for
the effective construction of such decompositions
(Algorithms~\ref{alg:simpleregularitydecomposition} and
\ref{alg:regularitydecomposition}) and proving its correctness
(Theorem~\ref{thm:algregdecomp}).

Our third and final main result concerns an old problem in the geometric
theory of differential equations.  There one usually considers only
\emph{regular differential equations}.  However, in many cases not even a
precise definition of this term is given and an effective test for
regularity is still unknown to the best of our knowledge, as it involves
considering not only one order, but all orders.  Hence, we first
provide a rigorous definition of this notion within our framework
(Definition~\ref{def:regde}) and then Theorem~\ref{thm:regde} asserts that
our algorithm for the construction of a regularity decomposition
automatically identifies in each irreducible component a Zariski dense
subset that is a regular differential equation.

This article is structured as follows.  In Sections~\ref{sec:cag} and
\ref{sec:cones}, we combine differential algebraic concepts with the
geometric theory of differential equations, leading to algebraic jet sets.
In Section~\ref{sec:singgde} we extend the classical definition of
singularities to arbitrary systems of differential equations, including
partial differential equations, and show that regular points are dense.
The subsequent Section~\ref{sec:regularitydecomposition} introduces our
concept of a regularity decomposition of a differential system and presents
an algorithm to compute this decomposition.  Then, Section~\ref{sec:regde}
looks at regular behaviour in prolongations and where it appears in our
decomposition.  Section~\ref{sec:exmp} treats some examples in detail.
Finally, some conclusions are given in Section~\ref{sec:conclusions}.

\section{Connecting Algebra and Geometry}
\label{sec:cag}

In this section, we lay the groundwork to formalise and effectively prove
the theorems of the later sections, by adapting and combining the geometric
theory of differential equations and methods from (differential) algebra.
For the convenience of the reader, we briefly summarise some basic concepts
of the geometric theory in Appendix~\ref{sec:geojet} and (differential)
algebra in Appendices~\ref{sec:algThomas_appendix} and \ref{sec:diffideal}.

This combination of methods represents a non-trivial task, as the
philosophies behind the used geometric and algebraic approaches are very
different.  In differential algebra, one always considers all orders
simultaneously by studying differential ideals. This implies that one has
to deal with infinitely many variables.  Such an approach is particularly
adapted to tackle completion questions, i.\,e.\ the construction of
hidden integrability conditions, for which it is unclear how geometric
approaches could be extended in the presence of singularities.\footnote{A
  fundamental problem arises already in the geometric definition of a
  prolonged equation, if the given equation is not a manifold but only a
  variety.  Thus basic notions like formal integrability or involution are
  highly non-trivial to generalise to equations admitting singularities and
  to our knowledge nobody has done this so far.}  By contrast, in the
geometric theory one works typically in a jet bundle of fixed order which
allows to define singularities as points with special properties, whereas
the Kolchin topology in differential algebra employs a rather generic
notion of points not suitable to describe singularities.

As our algebraic tools require that the underlying field is algebraically
closed, we consider throughout complex differential equations, i.\,e.\ all
variables are assumed to be complex-valued.  While the starting point of
the geometric theory is an arbitrary fibred manifold $\pi:\E\rightarrow\X$,
we consider exclusively trivial bundles with total space
$\E=\CC^{n}\times\CC^{m}$, base space $\X=\CC^{n}$ and $\pi$ the projection
on the first factor.  As all our work is of a local nature, this
restriction is not serious.  But it allows us to identify the total spaces
of the jet bundles $\je{\ell}$ with affine spaces $\AA_{\CC}^{d}$ of
suitable dimensions $d$ and thus apply standard concepts from algebraic
geometry to these spaces.  We use two topologies on $\je{\ell}$, namely the
Zariski topology and the standard topology induced by the Euclidean metric.
To avoid confusions, we will always write explicitly Zariski respectively
metric open or closed.

\begin{definition}\label{def:ade}
  An \emph{algebraic jet set} of order $\ell$ is a locally Zariski closed
  subset $\Jc{\ell}\subseteq\je{\ell}$ of a jet bundle of order $\ell$
  (i.\,e.\ the difference of two varieties in $\je{\ell}$).  It is an
  \emph{algebraic differential equation} of order $\ell$, if in addition
  the metric closure of $\pi^{\ell}(\Jc{\ell})$ is the whole base
  $\CC^{n}$.  An algebraic jet set or an algebraic differential equation is
  called \emph{irreducible}, if it is an irreducible locally Zariski closed
  subset.
\end{definition}

Compared with the classical geometric Definition~\ref{def:gde} of a
differential equation, varieties are used here instead of manifolds which
is simultaneously a generalisation and a restriction.  On one side, we
permit that the differential equation $\Jc{\ell}$ contains singular points
in the sense of algebraic geometry.  On the other side, we consider
exclusively differential equations which can be globally described as the
solution set of an algebraic system on $\je{\ell}$ with polynomials
$p_{i},q_{j}\in\D_{\ell}$ (see Appendices~\ref{sec:algThomas_appendix} and
\ref{sec:diffideal} for notations and definitions).

Definition~\ref{def:gde} furthermore requires that the restriction of the
canonical projection $\pi^{\ell}:\je{\ell}\rightarrow\X$ to the set
$\Jc{\ell}$ is a surjective submersion.  We are relaxing this requirement
in two directions: surjectivity is replaced by a closure condition for the
image and we do not impose a maximal rank condition.  The second relaxation
is crucial for the definition of geometric singularities.  Surjectivity of
the restricted projection represents a geometric way of saying that the
independent variables are indeed independent, as otherwise our differential
equation could imply relations between them.  However, this idea is also
captured by our condition on the metric closure of its image and for an
equation like $xu'=1$ surjectivity represents too strong a condition.  We
use the metric closure here instead of the Zariski one, as for the analysis
of the local solution behaviour around singularities (which we will not do
in this work) it is important that exceptional points may be considered as
the limit of a sequence of points in $\pi^{\ell}(\Jc{\ell})$.

In applications, the typical starting point is a differential system of the
form
$S = \{\, p_1 = 0, \, \ldots, \, p_s = 0, \, q_1 \neq 0, \, \ldots, \, q_t
\neq 0 \,\}$ as introduced in \eqref{eq:differentialsystem} in
Appendix~\ref{sec:diffideal} rather than an algebraic differential equation
as defined above.  Thus we start on the differential algebraic side and
discuss now how we can obtain geometric objects (and algebraic descriptions
of them).  It turns out that this process involves a number of subtleties
requiring a careful discussion.

We associate with such a differential system $S$ the \emph{differential}
ideal
\begin{equation*}
  \hat{\I}_{\mathrm{diff}}(S):=\langle p_1,\ldots,p_s\rangle_\Delta\subseteq\D
\end{equation*}
generated by the equations in $S$.  It induces for any order $\ell\in\NN_0$
the \emph{algebraic} ideal
\begin{equation*}
  \hat{\I}_{\ell}(S):=\hat{\I}_{\mathrm{diff}}(S)\cap\D_{\ell}\subseteq\D_\ell
\end{equation*}
as the corresponding finite-dimensional truncation.  Note that this ideal
automatically contains all hidden integrability conditions up to order
$\ell$.  The inequations in the differential system $S$ are also used to
define for any order $\ell\in\NN_0$ an \emph{algebraic}
ideal,\footnote{Note that it is pointless to introduce a
  \emph{differential} ideal defined by the inequations, as differentiating
  an inequation does not lead to a condition that has to be satisfied by
  any holomorphic or formal solution of the differential system $S$.}
however, in a slightly different manner:
\begin{equation*}
  \K_{\ell}(S):=\langle \hat{Q}_\ell\rangle_{\D_{\ell}}\qquad
  \mbox{with}\quad 
  \hat{Q}_\ell=\prod_{\substack{j=1\\ \ord{(q_{j})}\leq \ell}}^t q_j\,.
\end{equation*}
These ideals lead then to the algebraic jet sets
\begin{equation}\label{eq:Rhat}
  \Jhc{\ell}(S):=\Sol{\bigl(\hat{\I}_{\ell}(S)\bigr)}\setminus
                 \Sol{\bigl(\K_{\ell}(S)\bigr)}\subseteq\je{\ell}
\end{equation}
consisting of all points of $\je{\ell}$ satisfying both the equations and
the inequations in $S$ interpreted as algebraic equations in $\je{\ell}$.
Since their definition is based on the differential ideal
$\hat{\I}_{\mathrm{diff}}(S)$, these sets satisfy for any $k>0$ the
inclusions
$\pi^{\ell+k}_{\ell}\bigl(\Jhc{\ell+k}(S)\bigr)\subseteq \Jhc{\ell}(S)$.
In fact, we always have
$\pi^{\ell+k}_{\ell}\Bigl(\Sol{\bigl(\hat{\I}_{\ell+k}(S)\bigr)}\Bigr)=
\Sol{\bigl(\hat{\I}_{\ell}(S)\bigr)}$, but the inequations may lead to a
strict inclusion of the above jet sets \cite{lh:diffcountpoly}.

\begin{remark}\label{rem:l}
  While it is possible to define the ideals $\hat{\I}_{\ell}(S)$ and the
  algebraic jet sets $\Jhc{\ell}(S)$ for any order $\ell\in\NN_0$, these
  ideals and sets are really meaningful only if no equation $p_i$ in the
  underlying differential system is of an order greater than
  $\ell$. Assuming that the system $S$ is solvable and the sets
  $\Jhc{\ell}(S)$ are algebraic differential equations, their solution sets
  are otherwise not comparable, as all equations in $S$ of order greater
  than $\ell$ are ignored in the construction of $\Jhc{\ell}(S)$.  In
  particular, for different values of $\ell$ the corresponding equations
  $\Jhc{\ell}(S)$ may have different solution sets. Note that the orders of
  the inequations in $S$ are irrelevant here, as they should be more
  considered as conditions on allowed initial data. From now on, we
  always assume that $\ell$ is sufficiently large.
\end{remark}

While this construction of the algebraic jet sets $\Jhc{\ell}(S)$ appears
very natural, it faces a number of serious challenges making it inadequate
for our purposes:
\begin{enumerate}[(i),itemsep=3pt,parsep=0pt,topsep=3pt]
\item There may exist differential polynomials vanishing on every solution
  in $\Sol{(S)}$, but not contained in the differential ideal
  $\hat{\I}_{\mathrm{diff}}(S)$.
\item It is not so easy to study the algebraic jet sets $\Jhc{\ell}(S)$, as
  e.\,g.\ the ideals $\hat{\I}_{\ell}(S)$ are generally not radical---this
  is a consequence of (i)---and thus not the vanishing ideals of the
  underlying variety.  In particular, it is not immediately obvious whether
  the algebraic jet sets are non-empty. Furthermore, the algebraic jet sets
  $\Jhc{\ell}(S)$ are not necessarily algebraic differential equations, as
  it is not guaranteed that their projection
  $\pi^\ell\bigl(\Jhc{\ell}(S)\bigr)$ satisfies the closure condition of
  Definition~\ref{def:ade}.
\item The effective determination of bases for the algebraic ideals
  $\hat{\I}_{\ell}(S)$ is non-trivial, because of the possible existence of
  hidden integrability conditions.
\item The algebraic jet sets $\Jhc{\ell}(S)$ may be too small, as
  interpreting differential inequations as algebraic ones leads to a change
  in their semantics eliminating many ``interesting'' points.  Assume for
  simplicity that the system~$S$ contains the inequation $u_x\neq0$.  It
  entails that the $x$-derivative of any solution of $S$ can never be the
  zero \emph{function}.  Nevertheless, it is well possible that the
  $x$-derivative of a solution possesses zeros and thus the corresponding
  jets of this solution have a vanishing $u_x$-\emph{coordinate}.  However,
  no point on a set $\Jhc{\ell}(S)$ with $\ell>0$ can have a vanishing
  $u_x$-coordinate \cite{lh:diffcountpoly}.
\end{enumerate}

Challenge (i) requires a differential Nullstellensatz for differential
systems, i.\,e.\ an extension of Theorem~\ref{thm:dnss} that also includes
inequations.  \cite[Lemma 2.2.62]{dr:habil} asserts that the vanishing
ideal of $\Sol{(S)}$ is given by the differential ideal
\begin{equation}\label{eq:dnss}
  \I_{\mathrm{diff}}(S):=
  \sqrt{\hat{\I}_{\mathrm{diff}}(S) : \hat{Q}^\infty}\subseteq\D\qquad
  \mbox{with}\quad \hat{Q}=\prod_{j=1}^tq_j\,.
\end{equation}
Hence as first step we must replace the differential ideal
$\hat{\I}_{\mathrm{diff}}(S)$ by this ideal. However, using directly the
above definition of $\I_{\mathrm{diff}}(S)$ makes its explicit
determination rather expensive because of the required radical computation
(so that Challenge~(iii) becomes even more pronounced).

Our next step towards overcoming the mentioned difficulties consists of
restricting to \emph{simple} differential systems.  For any differential
system $S$, a differential Thomas decomposition provides us with simple
differential systems $S_1,\dots,S_k$ such that $\Sol(S)$ is the disjoint
union of the sets $\Sol(S_i)$.  Hence after such a decomposition we may
analyse instead of the original system $S$ one by one the simple systems
$S_1,\dots,S_k$. Recall, however, that such a decomposition is not unique.

So we assume from now on that $S$ is a simple differential system.  For
simple systems, \cite[Prop.~2.2.72]{dr:habil} entails that the ideal
$\I_{\mathrm{diff}}(S)$ defined in \eqref{eq:dnss} may alternatively be
constructed via a simple saturation without an explicit radical
computation:
\begin{equation}\label{eq:dnsssimp}
  \I_{\mathrm{diff}}(S)=\hat{\I}_{\mathrm{diff}}(S):Q^\infty\qquad
  \mbox{with}\quad Q=\prod_{i=1}^s\left(\init{(p_i)}\cdot\spt{(p_i)}\right)\,.
\end{equation}
Note that now we do not saturate with respect to the inequations in $S$ but
with respect to the product of the initials and separants of all the
equations in the differential system $S$.\footnote{Given an arbitrary
  differential system $S$, let $S_{1},\ldots,S_{k}$ be the simple systems
  of any differential Thomas decomposition of it.  Then
  \cite[Prop.~2.2.72]{dr:habil} yields the ideal decomposition
  \begin{displaymath}
    \I_{\mathrm{diff}}(S)=
    \bigcap_{i=1}^k\hat{\I}_{\mathrm{diff}}(S_{i}):Q_i^\infty
  \end{displaymath}
  where $Q_i$ is the product of the initials and separants of the equations
  in $S_i$. This intersection is in general not minimal, but no effective
  way is known to decide whether or not an ideal in this intersection is
  superfluous, which is again the so-called Ritt problem \cite[\textsection
  IV.9]{kol:diffalg}.}  As before, we use the differential ideal
$\I_{\mathrm{diff}}(S)$ to introduce for any sufficiently large order
$\ell$ (see Remark~\ref{rem:l}) the algebraic ideal
\begin{equation}
  \I_\ell(S):=\I_{\mathrm{diff}}(S)\cap\D_\ell\subseteq\D_\ell\,.
\end{equation}
Since the differential ideal $\I_{\mathrm{diff}}(S)$ is radical, the same
is true for all the finite truncations $\I_\ell(S)$ which greatly
simplifies the study of their varieties. Our steps so far suggest to
consider instead of the sets $\Jhc{\ell}(S)$ the algebraic jet sets
\begin{equation}\label{eq:R}
  \Jc{\ell}(S):=\Sol{\bigl(\I_{\ell}(S)\bigr)}\setminus
                \Sol{\bigl(\K_{\ell}(S)\bigr)}\subseteq\je{\ell}\,.
\end{equation}

\begin{lemma}
  Given a simple differential system $S$, these algebraic jet sets satisfy
  $\pi^{k+\ell}_{\ell}\bigl(\Jc{\ell+k}(S)\bigr)=\Jc{\ell}(S)$ for all
  prolongation orders $k>0$.
\end{lemma}

\begin{proof}
  As already mentioned above, the fact that the algebraic ideals
  $\I_{\ell}(S)$ stem from a differential ideal entails that
  $\pi^{k+\ell}_{\ell}\Bigl(\Sol{\bigl(\I_{\ell+k}(S)\bigr)}\Bigr)=
  \Sol{\bigl(\I_{\ell}(S)\bigr)}$.  Since we are now dealing with a simple
  differential system, no leader of an inequation is a derivative of a
  leader of an equation and the leaders of all equations and inequations
  are pairwise different. Hence we also have
  $\pi^{k+\ell}_{\ell}\Bigl(\Sol{\bigl(\K_{\ell+k}(S)\bigr)}\Bigr)=
  \Sol{\bigl(\K_{\ell}(S)\bigr)}$.
\end{proof}

Note that this result resembles the definition of formal integrability in
the geometric theory of differential equations
\cite[Def.~2.3.15]{wms:invol}.  However, many regularity assumptions are
made in the geometric theory and given a fibred submanifold
$\J_{\ell}\subseteq\je{\ell}$ its prolongation
$\J_{\ell+k}\subseteq\je{\ell+k}$ is defined via an intrinsic geometric
process.  Formal integrability is then a special property of some
submanifolds $\J_{\ell}$ encoding the absence of hidden integrability
conditions.  In our approach, it is an automatic consequence of the use of
a differential ideal and the simplicity of the defining differential
system.

\begin{remark}\label{rem:sat}
  From a geometric point of view, saturations as they appear in
  \eqref{eq:dnss} and \eqref{eq:dnsssimp}, respectively, have the following
  meaning: $\Sol(I:J^\infty)$ is the Zariski closure of the set
  $\Sol(I)\setminus\Sol(J)$. Thus, since the same ideal
  $\I_{\mathrm{diff}}(S)$ appears in \eqref{eq:dnss} and
  \eqref{eq:dnsssimp}, the variety $\Sol\bigl(\I_{\ell}(S)\bigr)$ is the
  Zariski closure of the set obtained by removing from
  $\Sol\bigl(\hat{\I}_{\ell}(S)\bigr)$ either all points at which a
  separant or an initial of an equation in the system $S$ vanishes or
  $\Sol(\K_\ell)$. In both cases, the Zariski closure restores many of
  the removed points. This is important for us, as most of the
  singularities we are interested in are actually such points.

  However, if a whole irreducible component of
  $\Sol\bigl(\hat{\I}_{\ell}(S)\bigr)$ consists only of such removed
  points, then it remains removed. Indeed, there are two possibilities for
  such a component. Either it does not define an algebraic differential
  equation on its own. Then it trivially cannot have any solutions and
  there is no point in looking for singularities. Or if it is an algebraic
  differential equation, then we analyse it elsewhere. Indeed, recall
  that we obtained a simple system only by computing a differential Thomas
  decomposition of our original system and the removed component
  corresponds to some other simple system arising in this decomposition.

  By \cite[Thm.~1.94]{lh:phd}, the ideal $\I_{\ell}(S)$ is furthermore
  equidimensional in the sense that all of its associated primes possess
  the same dimension which excludes in particular the existence of embedded
  prime components. This represents a further simplification entailed by
  the restriction to simple systems.
\end{remark}

\begin{remark}
  It follows from \cite[Cor.~1.96]{lh:phd} that the set of equations in
  any simple differential system forms a regular chain. Hence the ideals
  $\I(S)$ and $\I_\ell(S)$ are (differentially resp.\ algebraically)
  \emph{characterisable}, i.\,e.\ ideals defined by characteristic sets
  (cf.~\cite{eh:tri1,eh:tri2} for a survey of the properties of such
  ideals and \cite{lh:diffdimpoly} for an application).
\end{remark}

Even after this replacement, Challenge (iv) remains open and indicates that
we should enlarge the sets $\Jc{\ell}(S)$.  However, for a general
algebraic differential equation $\Jc{\ell}$ we face another challenge.  If
we consider the subset of $\Jc{\ell}$ obtained as the union of the images
of all prolongations $j_\ell\sigma$ of classical solutions of the equation,
then this subset may cover only a small part of $\Jc{\ell}$ (this happens in
particular, if hidden integrability conditions exist).  As one
of the main aspects of singularities is an analysis of the local solution
behaviour in their neighbourhood, we only want situations where this subset
lies dense in the considered differential algebraic equation.  This
motivates the following notion.

\begin{definition}\label{def:locint}
  The algebraic differential equation $\Jc{\ell}\subset\je{\ell}$ is
  \emph{locally integrable}, if $\Jc{\ell}$ contains a Zariski open and
  dense subset $\Rc{\ell}\subseteq\Jc{\ell}$ such that for every point
  $\rho\in\Rc{\ell}$ at least one classical solution $\sigma$ exists with
  $\rho\in\im{j_\ell\sigma}$.
\end{definition}

In general, it is difficult to decide whether a given algebraic
differential equation $\Jc{\ell}\in\je{\ell}$ is locally integrable, as
this obviously requires an existence theory for solutions. In particular,
such a decision cannot be made by a purely geometric analysis of
$\Jc{\ell}$, but requires the considerations of higher-order equations, too
(large parts of \cite{wms:invol} are concerned with this question in the
regular case). However, the situation is different under our assumption of
a simple differential system, as for such systems the local integrability
is essentially part of their definition.  More precisely, we obtain the
following result which already indicates how the above defined algebraic
jet sets $\Jc{\ell}(S)$ can be enlarged without losing this property.

\begin{proposition}\label{prop:locint}
  Let $S$ be a simple differential system with respect to a Riquier ranking
  and consider for an arbitrary
  order $\ell\in\NN$ the above defined algebraic jet set
  $\Jc{\ell}(S)$. Then its Zariski closure $\overline{\Jc{\ell}(S)}$ is a
  locally integrable algebraic differential equation.
\end{proposition}

\begin{proof}
  Obviously, $\Jc{\ell}(S)$ is Zariski dense in $\overline{\Jc{\ell}(S)}$
  and it suffices to prove that $\Jc{\ell}(S)$ is a locally integrable
  algebraic differential equation. The proof of the local integrability
  essentially boils down to an extension of Remark~\ref{rem:powerseries}
  where the construction of formal power series solutions is discussed.  We
  consider the Zariski open subset $\Rc{\ell}\subseteq\Jc{\ell}(S)$
  consisting of all smooth points at which no separant or initial of an
  equation in $S$ vanishes. By the considerations in Remark~\ref{rem:sat},
  $\Rc{\ell}$ is even Zariski dense in $\Jc{\ell}$. As remarked in
  \cite[Cor.~11]{glhr:tdds}, one can now straightforwardly adapt the proof
  of Riquier's Theorem~\ref{thm:riq} and conclude that the formal power
  series constructed in Remark~\ref{rem:powerseries} converges to a
  holomorphic solution $\sigma$ defined on some open subset of $\CC^n$.
\end{proof}

We are thus lead to consider the Zariski closure $\overline{\Jc{\ell}(S)}$
instead of $\Jc{\ell}(S)$.  Since it is a Zariski closed set in $\je{\ell}$
and thus a variety, we are obviously interested in its vanishing ideal.
Since $\I_\ell(S)$ is a radical ideal and we are working over an
algebraically closed field, it is a classical result in algebraic geometry
that it is given by the quotient ideal $\I_\ell(S):\hat{Q}_\ell$ (cf.\
e.\,g.\ \cite[Chapt.~4, Sect.~4, Thm.~7]{clo:iva}). The following lemma
shows that in our case this quotient simply means to ignore the inequations
in the system.

\begin{lemma}\label{lem:SolIl}
  For any order $\ell\in\NN$ we have
  $\overline{\Jc{\ell}(S)}=\Sol\bigl(\I_{\ell}(S)\bigr)$.
\end{lemma}

\begin{proof}
  Our assertion is equivalent to the following equality:
  \begin{displaymath}
    \biggl( \sqrt{ \hat{\I}_{\mathrm{diff}}(S): \hat{Q}^\infty } \cap
    \D_\ell \biggr):\hat{Q}_{\ell} =
    \sqrt{ \hat{\I}_{\mathrm{diff}}(S): \hat{Q}^\infty } \cap \D_\ell\,.
  \end{displaymath}
  The inclusion ``$\supseteq$'' is clear. For the reverse inclusion, we
  first note that, since $\hat{Q}_{\ell}$ divides $\hat{Q}$, we have
  $\hat{Q} = \hat{Q}_{\ell} \tilde{Q}$ for some $\tilde{Q}\in\D$. Let
  $P \in \D_\ell$ be such that
  $(P \hat{Q}_{\ell})^k \in \hat{\I}_{\mathrm{diff}}(S): \hat{Q}^\infty$
  for some positive integer $k$.  Then there exists an exponent $r\in\NN_0$
  such that
  $P^k\hat{Q}_{\ell}^k\hat{Q}^r\in\hat{\I}_{\mathrm{diff}}(S)$. Multiplication
  by $\tilde{Q}^k$ yields that
  $P^k\hat{Q}^{r+k}\in\hat{\I}_{\mathrm{diff}}(S)$. Hence
  $P^k\in\hat{\I}_{\mathrm{diff}}(S): \hat{Q}^\infty$ and thus $P$ lies in
  the radical.
\end{proof}

By definition, the equations in a simple differential system define a
passive system.  This observation allows us to resolve Challenge (ii).
Passivity implies consistency making it impossible that an equation $p_i$
depends only on the independent variables $x^j$. Hence for each algebraic
jet set $\Sol\bigl(\I_{\ell}(S)\bigr)$ it is clear that its image under the
canonical projection $\pi^\ell$ satisfies the closure condition of
Definition~\ref{def:ade} and thus that it is an algebraic differential
equation.  Furthermore, a passive system cannot contain a constant implying
via Hilbert's Nullstellensatz that all these sets are non-empty.

\begin{remark}\label{rmrk:genalgideal}
  The passivity of the equations also allows us to solve the remaining
  Challenge (iii): the explicit construction of generators for the
  algebraic ideals $\I_\ell(S)$ which we now use instead of
  $\hat{\I}_\ell(S)$.  The definition of passivity is based on the notion
  of (non-)multiplicative variables \cite{vpg:decomp,dr:habil}.  Consider
  now for any $\ell$ the following set
  \begin{multline}\label{eq:Bl}
    B_{\leq\ell}:=
    \Bigl\{\,\delta^\mu p_i\mid 1\leq i\leq s,\ |\mu|+\ord{(p_i)}\leq\ell,\\
    \mu_j=0
    \mbox{\ if\ }j\mbox{\ not Janet multiplicative for\ }p_i\,\Bigr\}
  \end{multline}
  obtained by differentiating each equation in $S$ with respect to its
  multiplicative variables until the order $\ell$ is reached. It provides
  us with an explicit generating set of the ideal $\hat{\I}_\ell(S)$.

  We define an algebraic system $S_{\leq\ell}$ by taking the elements of
  $B_{\leq\ell}$ as the equations and keeping all inequations of $S$ with
  order less than or equal to $\ell$. Since $S$ is assumed to be a simple
  differential system, it is easy to see that $S_{\leq\ell}$ is a simple
  algebraic system (both the initial and the separant of a derivative
  $\delta_kp_i$ are simply the separant of $p_i$). In
  \cite[Lemma~1.93]{lh:phd}, it is shown that
  $\I_{\mathrm{alg}}(S_{\leq\ell})=\I_\ell(S)$, where the ideal
  $\I_{\mathrm{alg}}(S_{\leq\ell})$ is defined in
  Equation~\eqref{eq:Ialg}. Recall from \eqref{eq:Ialg} that the
  determination of $\I_{\mathrm{alg}}(S_{\leq\ell})$ requires a
  saturation. Thus an explicit basis of $\I_\ell(S)$ is obtained by
  saturating the ideal generated by $B_{\leq\ell}$ by the product of the
  initials of the elements of $B_{\leq\ell}$.  This operation can be done
  effectively using Gr\"obner bases. It follows from Remark~\ref{rem:sat}
  and the definition \eqref{eq:Ialg} of $\I_{\mathrm{alg}}$ that
  $\overline{\Sol{(S_{\leq\ell})}}=\overline{\Jc{\ell}(S)}$.
\end{remark}

\begin{example}\label{ex:sat}
  To demonstrate in particular the effect of the saturation in the
  definition of the ideal $\I_{\mathrm{diff}}(S)$, we consider the
  following differential system consisting of two partial differential
  equations for an unknown function $u(x,y)$:
  \begin{equation}\label{eq:satex}
    p_1:= uu_x-yu-y^2 \,,\qquad  p_2:= yu_y-u\,.
  \end{equation}
  Adding the inequation $\spt{(p_1)}=u\not=0$ yields the only simple
  differential system~$S$ appearing in a differential Thomas decomposition
  of the system \eqref{eq:satex}. If we start with the differential ideal
  $\hat{\I}_{\mathrm{diff}}(S)=\langle p_1,p_2\rangle_{\Delta}$, then the
  algebraic ideal $\hat{\I}_{1}(S)=\hat{\I}_{\mathrm{diff}}(S)\cap\D_1$ has
  the prime decomposition
  $\hat{\I}_{1}(S)=\langle p_2,p_3\rangle\cap\langle u,y\rangle$ where
  \begin{equation}\label{eq:satp3}
    p_3:=u_yu_x-u-y
  \end{equation}
  and hence also $\hat{\I}_{\mathrm{diff}}(S)$ cannot be a prime. The
  saturation by $Q:=yu$ used in the definition \eqref{eq:dnsssimp} of
  $\I_{\mathrm{diff}}(S)$ removes the prime component $\langle u,y\rangle$
  of $\hat{\I}_{1}(S)$, more precisely
  $\I_{\mathrm{diff}}(S)=\langle p_2,p_3\rangle_{\Delta}$ and thus
  $\I_{1}=\langle p_2,p_3\rangle\subset\D_1$ (note that $p_1=yp_3-u_xp_2$).
  Indeed, if we compare for any order $\ell>0$ the algebraic jet sets
  $\Sol{(\I_\ell(S))}\subset \Sol{(\hat{\I}_\ell(S))}\subset \je{\ell}$,
  then we see that at all removed points the separants of the equations
  \eqref{eq:satex} vanish.

  In this particular case, the generators of the removed prime component do
  not define a consistent differential system, as one of them is the
  independent variable~$y$. Hence we are not losing any solutions by its
  removal. In other examples, we may remove components defining consistent
  systems. However, in such cases the properties of the differential Thomas
  decomposition ensure that the corresponding solutions appears in some
  other simple differential system.
\end{example}

\begin{remark}\label{rem:fintype}
  Riquier's Theorem~\ref{thm:riq} asserts that a certain initial value
  problem adapted to the choice of leaders in the equations of the system
  possesses a unique holomorphic solution (the explicit construction of the
  corresponding initial conditions is explained in more modern terms in
  \cite{sz:atls}; see also \cite[Sect.~9.3]{wms:invol}). If the system
  $S_{\leq\ell}$ is of finite type, then the coordinates of the considered
  point $\rho\in\Jc{\ell}$ provide all required initial data and in this
  case the holomorphic solution $\sigma$ such that
  $\rho\in\im{j_\ell\sigma}$ is uniquely determined. Otherwise, the
  coordinates of the considered point $\rho\in\Jc{\ell}$ provide only
  values for a finite subset of the infinitely many arbitrary Taylor
  coefficients of the series constructed in
  Remark~\ref{rem:powerseries}. Hence, in this case infinitely many
  different holomorphic solutions $\sigma$ exist such that
  $\rho\in\im{j_\ell\sigma}$, all of which possess the same Taylor
  expansion up to order $\ell$.
\end{remark}

\section{Vessiot Cones and Generalised Solutions}
\label{sec:cones}

In Appendix~\ref{sec:geojet}, we recall some basic concepts of Vessiot's
approach to a solution theory for differential equations.  Again some
adaptions are required, as we are now using a more general notion of
differential equations.  Furthermore, it turns out useful for the study
of singularities to introduce a more general concept of solutions than the
classical solutions of Definition~\ref{def:sol}.

The Vessiot space $\V_{\rho}[\Jc{\ell}]$ (cf.\ Definition~\ref{def:vessp})
at a point $\rho$ on a differential equation $\Jc{\ell}$ consists of the
tangential part of the contact distribution at~$\rho$.  As an algebraic jet
set $\Jc{\ell}$ is a locally Zariski closed subset which may contain
non-smooth points, the question arises how this definition should be
extended.  One could continue to apply it without changes using the tangent
space $T_{\rho}\Jc{\ell}$ in the sense of algebraic geometry.  Then one
would still obtain linear spaces; however, their dimension would be too
high.  We prefer therefore another extension.  Given a classical solution
$\sigma$ of $\Jc{\ell}$ such that $\rho\in\im{j_{\ell}\sigma}$, it follows
from a well-known characterisation of the \emph{tangent cone} as limit of
secants (see e.\,g.\ \cite[\textsection9.7, Thm.~6]{clo:iva}) that actually
$T_{\rho}(\im{j_{\ell}\sigma})\subseteq C_{\rho}\Jc{\ell}$ where
$C_{\rho}\Jc{\ell}$ denotes the tangent cone of $\Jc{\ell}$ at $\rho$.
This observation motivates the following extension of Vessiot spaces.

\begin{definition}\label{def:vessc}
  The \emph{Vessiot cone} $\V_{\rho}[\Jc{\ell}]$ of the algebraic jet set
  $\Jc{\ell}\subseteq\je{\ell}$ at a point $\rho\in\Jc{\ell}$ is the set
  $\V_{\rho}[\Jc{\ell}]=T_{\rho}\Jc{\ell}\cap\C_{\ell}|_{\rho}$.
\end{definition}

We continue to denote the family of all Vessiot cones by $\V[\Jc{\ell}]$.
At smooth points, the tangent cone and the tangent space coincide and
therefore we still speak of Vessiot spaces at such a point.  The following
elementary result recalls how Vessiot spaces can be easily computed at
smooth points.  As a consequence of it, we still often call $\V[\Jc{\ell}]$
the \emph{Vessiot distribution} of $\Jc{\ell}$ in the sequel, although,
strictly speaking, this terminology is not correct, as a cone is generally
not a linear space, but only a union of one-dimensional linear spaces.

\begin{proposition}\label{prop:vessdist}
  Let $\Jc{\ell}$ be an irreducible algebraic jet set.  Then the family of
  Vessiot cones $\V[\Jc{\ell}]$ define on a Zariski open and dense subset
  $\O_{V}\subseteq\Jc{\ell}$ a smooth regular distribution.
\end{proposition}

\begin{proof}
  The subset of all smooth points of $\Jc{\ell}$ is Zariski open and dense
  and defines a connected complex manifold \cite[Sect.~0.2]{gh:pag}.  At
  any point $\rho$ of this manifold, the tangent space $T_{\rho}\Jc{\ell}$
  and thus the Vessiot space $\V_{\rho}[\Jc{\ell}]$ can be computed using
  linear algebra.  As a locally Zariski closed set, the algebraic jet set
  $\Jc{\ell}$ is a Zariski open subset of the zero set of some polynomial
  functions $\Phi^{\tau}:\je{\ell}\rightarrow\CC$.  Since, by definition,
  the Vessiot spaces are contained in the contact distribution, we make for
  any vector $\vv\in\V_{\rho}[\Jc{\ell}]$ the ansatz
  \begin{align}\label{eq:ansvess}
    \vv=\sum_{i}a^{i}C_{i}^{(\ell)}|_{\rho}+
    \sum_{|\mu|=\ell}\sum_{\alpha}b^{\alpha}_{\mu}C^{\mu}_{\alpha}|_{\rho}
  \end{align}
  with yet to be determined coefficients $a^{i}$, $b^{\alpha}_{\mu}\in\CC$.
  At a smooth point $\rho$, such a vector is tangential to $\Jc{\ell}$, if
  and only if it satisfies $d\Phi^{\tau}|_{\rho}(\vv)=0$ for all~$\tau$.
  Hence, we obtain a homogeneous linear system for the coefficient vectors
  $\av$, $\bv$,
  \begin{align}\label{eq:vessdist}
    D(\rho)\av+M_{\ell}(\rho)\bv=0\,,
  \end{align}
  where the entries of the matrices $D$, $M_{\ell}$ are given
  by\footnote{The columns of the matrix $M_{\ell}$ are labelled by $\tau$
    and the rows by the pairs $(\mu,\alpha)$.}
  \begin{align}\label{eq:DM}
    D^{\tau}_{i}(\rho)=C_{i}^{(\ell)}(\Phi^{\tau})(\rho)\,,\qquad
    (M_{\ell})^{\tau\mu}_{\alpha}(\rho)=C^{\mu}_{\alpha}(\Phi^{\tau})(\rho)\,.
  \end{align}
  In general, the behaviour of (\ref{eq:vessdist}) varies over $\Jc{\ell}$;
  e.\,g.\ the dimension of $\V_{\rho}[\Jc{\ell}]$ may jump.  However,
  considered as functions of $\rho$, the solutions of \eqref{eq:vessdist}
  are smooth outside of a Zariski closed set and---by potentially enlarging
  this set---we may even assume that the dimension remains constant, being
  an upper semicontinuous function.  Thus on a Zariski open and dense set
  we obtain a smooth regular distribution.
\end{proof}

In an analogous way, we extend the notion of a symbol space to that of a
symbol cone.  Again it is straightforward to show that on a Zariski open
subset of $\Jc{\ell}$ the symbol spaces $\Nc{\ell}{\rho}$ define a smooth
regular distribution $\Nc{\ell}{}$.  At a smooth point $\rho\in\Jc{\ell}$,
the symbol space $\Nc{\ell}{\rho}$ consists of those solutions of
\eqref{eq:vessdist} for which all coefficients $\av$ vanish.  Hence, at
smooth points we can always decompose the Vessiot space as a direct sum of
linear subspaces,
\begin{align}\label{eq:Vdecomp}
\V_{\rho}[\Jc{\ell}]=\Nc{\ell}{\rho}\oplus\H_{\rho}\,,
\end{align}
with some $\pi^\ell$-transversal complement $\H_{\rho}$ which is not uniquely
determined.

\begin{remark}\label{rem:vessprol}
  If one computes for a differential equation $\Jc{\ell}$
  order by order a formal power series solution around some expansion
  point, then one obtains for the Taylor coefficients of order $\ell+1$ an
  inhomogeneous linear system with a matrix and right hand side depending
  on the lower order coefficients (see \cite[Sect.~2.3]{wms:invol} for more
  details).  One can show that the linear system (\ref{eq:vessdist}) is a
  homogenised form of this linear system \cite[Rem.~9.5.6]{wms:invol}.  Let
  us assume that it is possible to solve (\ref{eq:vessdist}) in such a way
  that the coefficients $\av$ remain undetermined (this is actually what we
  expect to happen generically).  Then we can relate the solutions of
  (\ref{eq:vessdist}) with the derivatives $u^{\alpha}_{\nu}$ of order
  $\ell+1$ of the power series solution.  Indeed, in this case we must find
  for each value $1\leq i\leq n$ a solution $\bar\av$, $\bar\bv$ such that
  $\bar a^{j}=\delta^{j}_{i}$ and
  $\bar b^{\alpha}_{\mu}=u^{\alpha}_{\mu+1_{i}}$.  Conversely, one can see
  that if no such solution exists for (\ref{eq:vessdist}) at some point
  $\rho\in\Jc{\ell}$, then no \emph{smooth} solution $\sigma$ with
  $\rho\in\im{j_\ell\sigma}$ can exist, as at least one derivative of order
  $\ell+1$ becomes infinite.
\end{remark}

In the decomposition \eqref{eq:Vdecomp}, we can choose at any smooth point
$\rho\in\Jc{\ell}$ an arbitrary complement $\H_\rho$.  
A solution $\sigma$ with
$\rho\in\im{j_ \ell\sigma}$ can exist only, if the complement $\H_\rho$ is
$n$-dimensional (cf.\ Proposition~\ref{prop:sol}).  This raises the question
whether it is possible to correlate the choices in the neighbourhood of a
point in such a way that the chosen complements form an involutive
distribution.  If this is possible at all, then for most systems, there are
actually infinitely many ways to do this (parametrised by the symbol).
Only for a special class of differential equations---comprising in
particular most ordinary differential equations---only a unique possibility
exists.

\begin{definition}\label{def:fintype}
  An algebraic differential equation $\Jc{\ell}$ is \emph{of finite type},
  if it contains a Zariski open and dense subset
  $\F_\ell\subseteq\Jc{\ell}$ such that at all points $\rho\in\F_\ell$ the
  symbol cone $\Nc{\ell}{\rho}$ vanishes.
\end{definition}

In the literature, one can find many alternative names for equations of
finite type.  In the theory of linear systems, the term \emph{holonomic}
system is very popular.  Another common terminology, in particular for
partial differential equations, is \emph{maximally overdetermined} system.
From a geometric point of view, (regular first-order) equations of finite
type correspond to connections over the fibration $\pi$ (see
\cite[Remark~2.3.6]{wms:invol}).

For the analysis of singularities, it turns out to be convenient to
introduce more general kinds of solutions directly as geometric objects
without reference to a section.  The following definition simply relaxes
some of the conditions on the subdistribution $\H$ in the second part of
Proposition~\ref{prop:sol}.  Note that such a generalised solution lives in
the jet bundle $\je{\ell}$ and not in the total space $\E$ of the fibration
$\pi$ like a section, but it can be projected to $\E=\je{0}$.

\begin{definition}\label{def:gensol}
  Let $\Jc{\ell}\subseteq\je{\ell}$ be an algebraic differential equation
  in $n$ independent variables.  A \emph{generalised solution} of
  $\Jc{\ell}$ is an $n$-dimensional submanifold $\N\subseteq\Jc{\ell}$ such
  that $T\N\subseteq\V[\Jc{\ell}]|_{\N}$.  A \emph{geometric solution} of
  $\Jc{\ell}$ is the projection $\pi^{\ell}_{0}(\N)\subseteq\E$ of a
  generalised solution $\N$.
\end{definition}

If the section $\sigma:\X\rightarrow\E$ defines a solution of $\Jc{\ell}$,
then $\im{j_{\ell}\sigma}$ is automatically a generalised solution with
$\im{\sigma}$ as the corresponding geometric solution; this follows
immediately from the definition of the Vessiot distribution. However, if
the differential equation $\Jc{\ell}$ has geometric singularities as
defined below, then not every geometric solution is the image of a section
$\sigma:\X\rightarrow\E$ (in fact, generally it is not even a manifold).

\section{Singularities of General Differential Equations}
\label{sec:singgde}

In classical analysis, one usually studies singularities like a blow-up or
a shock.  Thus the singular behaviour refers to an individual solution and
consists of either the solution itself or some derivative of it becoming
infinite at some finite point $\xv\in\X$. By contrast, we study
singularities of the differential system $S$ itself: we define
singularities as points $\rho\in\Jc{\ell}$ for some sufficiently high
order~$\ell$ such that generalised solutions in the sense of
Definition~\ref{def:gensol} in the neighbourhood show a ``special''
behaviour.  If $\Jc{\ell}$ is a differential equation of finite type, then
we expect that on any sufficiently small neighbourhood of a regular point
$\rho\in\Jc{\ell}$ a unique foliation of the neighbourhood by generalised
solutions exists and that all generalised solutions are the image of
prolonged classical solutions.  If the equation is not of finite type, then
around regular points still such foliations exist, but they are no longer
unique. In fact, infinitely many foliations exist.

\begin{definition}\label{def:regsing}
  Let $\Jc{\ell}\subseteq\je{\ell}$ be a locally integrable algebraic
  differential equation in $n$ independent variables. A non-smooth point
  $\rho\in\Jc{\ell}$ is called an \emph{algebraic singularity} of
  $\Jc{\ell}$. A smooth point $\rho\in\Jc{\ell}$ is called
  \begin{enumerate}	
  \item[(i)] \emph{regular}, if a metric open neighbourhood
    $\rho\in\U\subseteq\Jc{\ell}$ exists such that the Vessiot distribution
    $\V[\Jc{\ell}]$ is regular on $\U$ and can be decomposed as
    $\V[\Jc{\ell}]=\Nc{\ell}{}\oplus\H$ with an $n$-dimensional,
    transversal, involutive, smooth distribution $\H\subseteq T\U$;
  \item[(ii)] \emph{regular singular}, if a metric open neighbourhood
    $\rho\in\U\subseteq\Jc{\ell}$ exists such that the Vessiot distribution
    $\V[\Jc{\ell}]$ is regular on $\U$ but at the point $\rho$ no
    $n$-dimensional complement to the symbol $\Nc{\ell}{\rho}$ exists,
    i.\,e.\ we have $\dim{\V_{\rho}[\Jc{\ell}]}-\dim{\Nc{\ell}{\rho}}<n$;
  \item[(iii)] \emph{irregular singular}, if the Vessiot spaces do not form
    a regular distribution on any metric open neighbourhood
    $\rho\in\U\subseteq\Jc{\ell}$; i.\,e.\ any such neighbourhood contains
    at least one point $\bar\rho$ such that
    $\dim{\V_{\bar\rho}[\Jc{\ell}]}<\dim{\V_{\rho}[\Jc{\ell}]}$.
  \end{enumerate}
  An irregular singularity $\rho\in\Jc{\ell}$ is called \emph{purely
    irregular}, if an $n$-dimensional complement to the symbol space
  $\Nc{\ell}{\rho}$ exists, i.\,e.\
  $\dim{\V_{\rho}[\Jc{\ell}]}-\dim{\Nc{\ell}{\rho}}=n$.
\end{definition}

The notion of a purely irregular singularity is new and becomes necessary
only for equations not of finite type.  At generic singularities such a
distinction is not necessary: generically the dimension of the symbol space
$\Nc{\ell}{\rho}$ jumps at a singularity only by one and in this case any
irregular singularity is automatically purely irregular. At points where no
$n$-dimensional transversal complement to the symbol space exists, not even
a formal power series solution can exist. Hence pure irregularity is
important for any kind of solution theory around singularities.

\begin{example}
  As a concrete example where all different types of points appearing in
  the above definition occur, we consider the following second-order system
  of semilinear partial differential equations for one unknown function $u$
  in two independent variables $x$, $y$:
  \[
    \begin{aligned}
        x^2u_{xx} + xu_x + (x-1)^2u &= 0\,,\\
        (1-y^2)u_{yy} + 2yu_y + 2u &= 0\,.
    \end{aligned}
  \]
  If we consider the algebraic differential equation
  $\mathcal{J}_2\subset J_2\pi$ defined by it, then one must distinguish
  seven different cases in the analysis of the linear system defining the
  Vessiot spaces:
  \begin{enumerate}
  \item \emph{Regular points} on $\mathcal{J}_2$ are characterised by the
    conditions $x\neq0$ and $y^2-1\neq0$. They have a three-dimensional
    Vessiot space.
  \item Points where $x=0$, $y^2-1\neq0$ and either $u_x\neq0$ or
    $u_y\neq0$ are \emph{regular singular}. They also possess a
    three-dimensional Vessiot space. As the coefficients $a_1$ and $a_2$ in
    (\ref{eq:ansvess}) must satisfy the equation $2u_xa_1+u_ya_2=0$, only a
    one-dimensional transversal complement exists.
  \item Basically the same holds for points where $y^2-1=0$, $x\neq0$ and
    either $yu_x+u_{xy}\neq0$ or $u\neq0$: they are \emph{regular singular}
    and have a three-dimensional Vessiot space with a one-dimensional
    transversal complement defined by the equation
    $(yu_x+u_{xy})a_1-2ua_2=0$.
  \item Points where $x=0$, $y^2-1=0$ and either $u_x\neq0$ or
    $yu_{xy}+u_x\neq0$ are \emph{irregular singularities} which are not
    purely irregular: the Vessiot space is four-dimensional with a
    one-dimensional transversal complement defined by the condition
    $a_1=0$.
  \item Points where $x=0$, $u_x=0$, $u_y=0$ and $y^2-1\neq0$ are
    \emph{purely irregular singular} and possess a four-dimensional Vessiot
    space defined by the equation $(y^2-1)b_{02}-2yu_{xy}a_1=0$.
  \item The same behaviour is shown by points with $y^2-1=0$, $u=0$,
    $u_y=0$, $x\neq0$, but with the Vessiot space defined by the equation
    $x^2b_{20}+(x^2-xy-2x-1)u_xa_1=0$.
  \item Finally, the points where $x=0$, $y^2-1=0$, $u_xy=0$ and $u=0$ are
    also \emph{purely irregular singular} but now with a five-dimensional
    Vessiot space.
  \end{enumerate}
  Note that the cases 2, 3 and 4 do not correspond to an algebraic jet set
  but the union of two such sets, because of the disjunctions in their
  defining conditions.  Hence, if one applies the algorithm we present in
  the next section to this example, then one obtains actually $10=7+3$
  cases.
\end{example}


\begin{remark}\label{rem:singfintype}
  For differential equations of finite type (and thus in particular for all
  not underdetermined ordinary differential equations),
  Definition~\ref{def:regsing} can be considerably simplified, as it is no
  longer necessary to consider neighbourhoods. For a passive equation of
  finite type, it is a priori clear that the expected dimension of the
  Vessiot space at a regular point is $n$. Thus singularities can be
  recognised by a simple comparison with this value (see \cite{wms:aims}
  for such a definition of regular and irregular singularities). Our more
  complicated approach via neighbourhoods is the prize to be paid for the
  fact that Definition~\ref{def:regsing} is to our knowledge the first
  attempt to provide a systematic taxonomy of the singularities of
  arbitrary systems of partial differential equations. We do not claim that
  our definition provides already a complete taxonomy, however it appears
  very natural from the point of view of the geometric theory of
  differential equations, as it takes all fundamental geometric objects
  (Vessiot and symbol spaces) into account.
\end{remark}

\begin{remark}\label{rem:regular}
  If we ignore the requirement that in the neighbourhood of a regular point
  the complement $\H$ must be involutive, then the three cases in
  Definition~\ref{def:regsing} correspond to the analysis of the linear
  system \eqref{eq:vessdist}. A necessary condition for a point
  $\rho\in\Jc{\ell}$ to be regular is that the symbol matrix $M_\ell(\rho)$
  has its maximal possible rank and that this rank coincides with the
  maximal possible rank of the augmented matrix
  $\bigl(D(\rho)\mid M_\ell(\rho)\bigr)$.  At a regular singular point, the
  augmented matrix has still the maximal possible rank, but the rank of the
  symbol matrix has dropped.  At an irregular singular point even the rank
  of the augmented matrix has dropped.

  In the case of ordinary differential equations, any complement $\H$ can
  only be one-di\-men\-sio\-nal and thus is trivially involutive wherever it
  defines a regular distribution.  Hence for ordinary differential
  equations Definition~\ref{def:regsing} provides a complete taxonomy of
  all points on $\Jc{\ell}$.  For partial differential equations, it is in
  general difficult to prove the involutivity of $\H$ around points where
  the above mentioned necessary condition for a regular point is satisfied.

  We always study algebraic jet sets coming from simple
  differential systems produced by a differential Thomas decomposition.
  Here it is possible to prove for generic points that they are
  regular.  For the other points satisfying the above necessary condition,
  two possibilities arise. If they are regular (which we are not able to
  prove), then there are (prolonged) solutions going through them. By the
  properties of the differential Thomas decomposition, they must belong to
  the solution set of another simple differential system arising in the
  decomposition. Hence one can argue that they are irrelevant in the
  analysis of the given simple differential system.  If they are not
  regular, then they fall outside the taxonomy of Definition~\ref{def:regsing}.
  It is unclear whether this case is actually possible;
  at least we do not know of any concrete example where such points
  appear. They could be related to novel kinds of singular behaviour that
  only exist in partial differential equations, but they also could simply
  be accidentially introduced by taking the Zariski closure.
\end{remark}

\begin{remark}\label{rem:geosing}
  Regular and irregular singularities may be considered as \emph{geometric
    singularities} in the sense that they represent the critical points of
  the restriction of the canonical projection
  $\pi^{\ell}_{\ell-1}:\je{\ell}\rightarrow\je{\ell-1}$ to the considered
  subset $\Jc{\ell}$, i.\,e.\ of the map
  $\hat{\pi}^\ell_{\ell-1} :
  \Jc{\ell}\rightarrow\pi^\ell_{\ell-1}(\Jc{\ell}) \subseteq \je{\ell-1}$.
  In other words, they are the points $\rho$ where the tangent map
  $T_{\rho}\hat{\pi}^\ell_{\ell-1}$ is not surjective.  Indeed, at smooth
  points the symbol spaces are the kernels of the restricted projection
  $\hat{\pi}^\ell_{\ell-1}$. Hence, one may say that geometric
  singularities are those points where the dimension of the symbol space
  jumps. This is the classical approach to define singularities of implicit
  ordinary differential equations, as one can find it e.\,g.\ in
  \cite{via:geoode}.
\end{remark}

Definition~\ref{def:regsing} is really meaningful only, if we can show that
the regular points form a Zariski dense subset and thus really represent
the ``regular'' behaviour.  The main problem in proving this fact consists
in establishing the existence of a smooth distribution $\H$ possessing all
the required properties.  As this is much easier for systems of finite
type, we treat this case separately.

\begin{proposition}\label{prop:regdensefintype}
  Let $S$ be a simple differential system, with respect to a Riquier
  ranking, comprising no equation of an
  order greater than $\ell\in\NN$ for which the associated algebraic
  differential equation $\Jc{\ell}(S)$ defined in (\ref{eq:R}) is of finite
  type.  Then the regular points in its Zariski closure
  $\overline{\Jc{\ell}(S)}$ contain a Zariski open and dense subset.
\end{proposition}

\begin{proof}
  By Proposition~\ref{prop:locint}, $\overline{\Jc{\ell}(S)}$ is a locally
  integrable algebraic differential equation. In the proof of that
  proposition it was shown that every point $\rho$ in a Zariski open and
  dense subset $\Rc{\ell}\subseteq\Jc{\ell}(S)$ lies in the image of a
  prolonged classical solution~$\sigma$. In Remark~\ref{rem:fintype}, it
  was discussed that for an equation of finite type this solution~$\sigma$
  is uniquely determined by~$\rho$. This implies in particular that for
  different such solutions the images of their prolongations cannot
  intersect in a sufficiently small neighbourhood of~$\rho$. Hence these
  images define a foliation of such a neigbourhood with $n$-dimensional
  leaves and the tangent spaces of the points on the leaves are just the
  Vessiot spaces there. This observation implies that the Vessiot
  distribution restricted to this neighbourhood is integrable and hence by
  the Frobenius Theorem involutive. Therefore all smooth points
  $\rho\in\Rc{\ell}$ are regular in the sense of
  Definition~\ref{def:regsing}.
\end{proof}

Note that this proof also tells us precisely the local solution behaviour
near a regular point: the neighbourhood of the point is foliated by
$n$-dimensional transversal leaves which are generalised solutions
projecting on geometric solutions which are the images of classical
solutions. The generalisation to arbitrary systems requires the use of the
Vessiot theory of differential equations introduced originally in
\cite{ves:int}.  A modern presentation relating it to the geometric theory
of differential equations can be found in \cite{df:diss} (see also
\cite{wms:vessconn2} or \cite[Sects.~9.5/6]{wms:invol}).  These references
are concerned with the existence of flat Vessiot connections. The
horizontal bundle of such a connection is nothing but a smooth distribution
$\H$ with all the properties required in the definition of a regular point.

\begin{theorem}\label{thm:regdense}
  Let $S$ be a simple differential system, with respect to a Riquier
  ranking, comprising no equation of an order greater than $\ell\in\NN$ and
  $\Jc{\ell}(S)$ the associated algebraic differential equation defined in
  (\ref{eq:R}). Then the regular points in the Zariski closure
  $\overline{\Jc{\ell}(S)}$ contain a Zariski open and dense subset.
\end{theorem}

\begin{proof}
  As in the proof of Proposition~\ref{prop:regdensefintype}, we consider
  again the Zariski open and dense subset $\Rc{\ell}\subseteq\Jc{\ell}(S)$.
  As a smooth point, any point $p\in\Rc{\ell}$ lies on exactly one
  irreducible component of $\Jc{\ell}(S)$.  The intersection of $\Rc{\ell}$
  with this irreducible component is a manifold which, by the proof of
  Proposition~\ref{prop:locint}, defines a formally integrable differential
  equation in the sense of the geometric theory, since local integrability
  trivially entails formal integrability.  The equations in a simple system
  form by definition a (differential) Janet basis and it is easy to see
  that consequently their principal parts (introduced in
  Appendix~\ref{sec:geojet}) define at any point $\rho\in\Rc{\ell}$ a
  (polynomial) Janet basis of the principal symbol module $\M[\rho]$.  The
  maximal degree of a generator in this basis is at most $\ell$. By
  \cite[Thm.~5.4.12, Rem.~5.4.13]{wms:invol}, this Janet basis induces a
  free resolution of $\M[\rho]$ and the form of this resolution implies
  that the Castelnuovo-Mumford regularity of $\M[\rho]$ is at most $\ell$.
  By \cite[Rem.~6.1.23]{wms:invol}, this implies that the symbol
  $\Nc{\ell}{\rho}$ is involutive at any point $\rho\in\Rc{\ell}$, as the
  order at which a symbol becomes involutive is determined by the
  regularity of the principal symbol module.  Hence the manifold defines
  even an involutive differential equation in the sense of the geometric
  theory.  Now \cite[Thm.~3]{wms:vessconn2} (or equivalently
  \cite[Thm.~9.6.11]{wms:invol}) asserts the existence of a smooth
  distribution $\H$ with the required properties in a neighbourhood of $p$,
  so that $p$ is indeed a regular point.
\end{proof}

It should be emphasised that the distribution $\H$ appearing at the end of
the proof is never unique for a system which is not of finite type. Again,
each particular choice of a distribution $\H$ induces a foliation of a
neighbourhood of the regular point with $n$-dimensional transversal leaves
which are the images of generalised solutions coming from classical
solutions. However, for a system not of finite type there always exist
infinitely many such choices and hence infinitely many different
foliations. Nevertheless, we may still say that regular points are
characterised by the existence of at least one such foliation.

\begin{example}\label{exmp:clairaut}
  It should be noted that the notions introduced in
  Definition~\ref{def:regsing} are relative in the sense that they
  obviously depend on the choice of the differential equation
  $\Jc{\ell}$. In some situations one may have more than one option and
  then obtains different results for certain points. As a simple concrete
  example, we may consider the Clairaut equation $u=xu'+f(u')$; the
  corresponding algebraic jet set is shown in Figure~\ref{fig:clair} in
  blue. It represents a classical instance of a differential equations with
  a singular integral. Its general solution is given by the straight lines
  $u(x)=cx+f(c)$ with a parameter $c$ (some lines are shown in green in the
  figure). Their envelope is the singular integral given parametrically by
  $x(\tau)=-f'(\tau)$, $u(\tau)=-\tau f'(\tau)+f(\tau)$ (shown in red). The
  singular integral is the sole solution of the overdetermined system
  $u=xu'+f(u')$ and $f'(u')+x=0$ (the separant of the first equation). If
  we choose as $\Jc{1}$ the whole blue surface, then all points on the
  singular integral are irregular singularities, as the Vessiot spaces are
  two-dimensional there. If we choose instead only the curve defined by the
  prolonged singular integral (which represents an algebraic differential
  equation in its own right\footnote{With the notation from below, the
    Clairaut equation can be decomposed into two primary components: the
    general solution and the singular integral. This decomposition can be
    seen when prolonging to order two.}), then all points on it are
  regular, as now for the overdetermined system the Vessiot space is always
  one-dimensional and coincides with the tangent space of the curve.  This
  effect is captured in Definition~\ref{def:regsing} by the use of a metric
  open neighbourhood of the considered point. Depending on the choice of
  $\Jc{\ell}$, the dimension of the neighbourhood as a smooth manifold may
  vary and the neighbourhood decides what is considered as regular and what
  as singular.

\begin{figure}[htb]
   \centering
   \includegraphics[height=5cm]{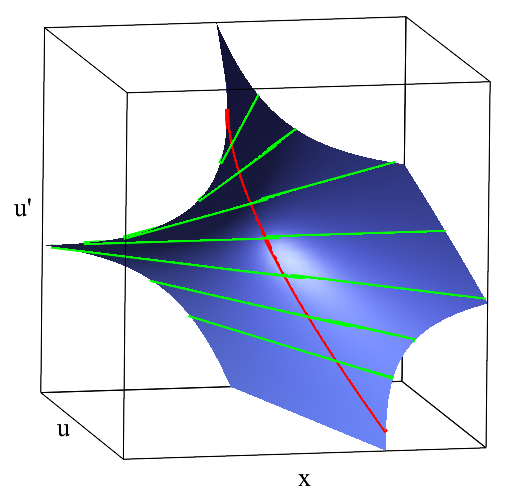}\quad
   \includegraphics[height=5cm]{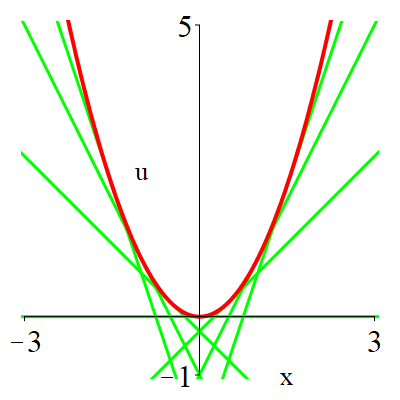}
   \caption{Clairaut equation for $f(s)=-\tfrac{1}{4}s^2$ with the singular
     integral in red. Left: generalised solutions in $J_{1}\pi$. Right:
     solution graphs in $x$-$u$ plane.\label{fig:clair}}
\end{figure}

\end{example}

\section{Regularity Decomposition of a Differential System}
\label{sec:regularitydecomposition}

The geometric theory of differential equations considers usually
exclusively ``regular'' equations, although it is not so easy to provide a
rigorous definition of what this regularity should be and even harder to
verify effectively whether or not a given equation is regular.  Very often,
one only finds generic statements that all assertions are valid outside of
some (unspecified) hypersurface (see e.\,g.\ \cite{bm:sdi} and references
therein).  We now define first a rigorous notion of a regular algebraic jet
set.  For such jet sets, we can extend the pointwise decomposition
(\ref{eq:Vdecomp}) to a global one: $\V[\Jc{\ell}]= \Nc{\ell}{}\oplus\H$
with some smooth vector bundle $\H$.  We study the generalisation to a
regular differential equation in the sense of the geometric theory of
differential equations in the following Section~\ref{sec:regde}.

\begin{definition}\label{def:regjetset}
  An algebraic jet set $\Jc{\ell}\subseteq\je{\ell}$ is
  \emph{regular}, if
  \begin{enumerate}[label=\textup{(\roman*)},
                    itemsep=0pt,parsep=0pt,topsep=0pt]
  \item it consists only of smooth points, i.\,e.\ $\Jc{\ell}$ is a
    smooth manifold,
  \item its Vessiot distribution $\V[\Jc{\ell}]$ defines a smooth
    vector bundle over $\Jc{\ell}$ and
  \item its symbol $\Nc{\ell}{}$ defines a smooth vector bundle over
    $\Jc{\ell}$. 
  \end{enumerate}
\end{definition}

Let $S$ be a differential system.  As discussed in Section~\ref{sec:cag},
as a first step we compute a differential Thomas decomposition of $S$ into
simple differential systems each of which we then treat
separately. Thus we assume from now on that $S$ is already a simple
differential system. We choose a sufficiently high order $\ell$ and
consider the associated algebraic jet set
$\overline{\Jc{\ell}(S)}\subset\je{\ell}$.  In general, it might be a
reducible variety. As any point contained in the intersection of two
irreducible components of $\overline{\Jc{\ell}(S)}$ is automatically an
algebraic singularity, we prefer to study each irreducible component
separately. We then want to express each irreducible component as a
disjoint union of \emph{regular} algebraic jet sets.

\begin{definition}\label{def:regdecomp}
  Let $S\subset\D$ be a simple differential system and
  $\overline{\Jc{\ell}(S)}\subset\je{\ell}$ the associated algebraic jet
  set in a sufficiently high order~$\ell$. Let furthermore
  $\overline{\Jc{\ell}(S)}=\Jc{\ell,1}\cup\cdots\cup\Jc{\ell,t}$ be its
  decomposition into irreducible varieties.  A \emph{regularity
    decomposition} of the variety $\Jc{\ell,k}$ represents it as a disjoint
  union of finitely many regular algebraic jet sets
  $\Jcp{\ell,k}{1},\dots,\Jcp{\ell,k}{r}$, the \emph{regularity components}
  of $\Jc{\ell,k}$, and of the set
  $\mathrm{ASing}\bigl(\overline{\Jc{\ell}(S)}\bigr)$ of algebraic
  singularities.
\end{definition}

If we classify the points on the irreducible variety $\Jc{\ell,k}$
according to Definition~\ref{def:regsing}, then if one point on a
regularity component $\Jcp{\ell,k}{i}$ is a regular (irregular)
singularity, then all other points on this component are regular
(irregular) singularities, too.  Indeed, Definition~\ref{def:regjetset}
implies that at all points on a regular algebraic jet set the symbol and
the Vessiot space, respectively, have the same dimension. The situation is
more involved for regular points (of partial differential equations) as
discussed in Remark~\ref{rem:regular}.  However, as a consequence of
Remark~\ref{rem:sat} and Theorem~\ref{thm:regdense}, we may conclude that
the regular points contain on each prime component $\J_{\ell,k}$ a Zariski
open and dense subset.  If a point lies in the intersection of several
irreducible components, then we classify it separately with respect to each
of these components. It is well possible that one obtains here different
results (see Examples~\ref{exmp:clairaut} or \ref{ex:sphereplane} for
concrete instances).

We now prove the existence of regularity decompositions by providing an
algorithm for their construction. For at least one component of the
obtained decomposition (which contains a Zariski open and dense subset), we
prove that it consists of regular points. As discussed in
Remark~\ref{rem:regular}, we cannot exclude the possibility that in some
other regularity components the Vessiot and the symbol space have at all
points the dimension expected for a regular point, but we are unable to
prove the involutivity of the complement $\H$. We say that such points are
of \emph{unknown type}.

The first step of our algorithm consists of determining a generating set
$\{ p_1, \ldots, p_s \}$ of the algebraic ideal $\I_\ell(S)$ according to
Remark~\ref{rmrk:genalgideal}. As second step, we determine the minimal
prime decomposition $\I_\ell(S)=\bigcap_{k=1}^t\I_{\ell,k}$ of the ideal
$\I_\ell(S)$ which is radical by definition.  According to
Lemma~\ref{lem:SolIl}, $\overline{\Jc{\ell}(S)}=\Sol{(\I_\ell(S))}$ and, by
construction, $\Sol{(\I_\ell(S))}=\bigcup_{k=1}^t\Sol{(\I_{\ell,k})}$.  We
then determine for each irreducible component $\Sol{(\I_{\ell,k})}$
separately a regularity decomposition.

For the determination of these regularity decompositions, we exploit that
our taxonomy of regular and singular points (Definition~\ref{def:regsing})
is mainly based on the properties of the linear system \eqref{eq:vessdist}
determining the Vessiot distribution.  If
$\{ p_{k,1}, \ldots, p_{k,s_k} \}$ is a generating set of the prime
component $\I_{\ell,k}$, then we use these polynomials for setting up the
linear system \eqref{eq:vessdist}, as it simply encodes a condition of
tangency to the irreducible component $\Jc{\ell,k}=\Sol{(\I_{\ell,k})}$.

In addition, we set up a second linear system for the detection of the
algebraic singularities defined by
\begin{equation}\label{eq:jacobian}
  \JJ(p_{k,r}):=
  \left( \sum_{|\mu| = \ell} \sum_{\alpha} c_{\mu}^{\alpha}
    \partial_{u^{\alpha}_{\mu}}  + \sum_j d^j \partial_{x^j} \right) p_{k,r}=0, 
  \quad r=1,\dots ,s_k.
\end{equation}
The left hand side is obtained by multiplying the Jacobian matrix of
$p_{k,1},\dots, p_{k,s_k}$ by the vector of auxiliary indeterminates
$c_{\mu}^{\alpha}$ and $d^j$. The equations in the combined linear system
may be considered as elements of the extended polynomial ring
$\mathcal{D}_{\ell}^{{ \rm\, ex}}=\mathcal{D}_{\ell}[\av,\bv,\cv,\dv]$
where we have adjoined the auxiliary indeterminates $\av$, $\bv$ of the
ansatz \eqref{eq:ansvess} and $\cv$, $\dv$ of
\eqref{eq:jacobian}. Furthermore, we consider this combined linear system
only at points on $\Jc{\ell,k}$ and thus add the equations
$p_{k,1}, \ldots, p_{k,s_k}\in
\mathcal{D}_{\ell}\subset\mathcal{D}_{\ell}^{{\rm\,ex}}$. We compute an
algebraic Thomas decomposition of the combined system in
$\mathcal{D}_{\ell}^{{\rm\,ex}}$ using an ordering satisfying the following
conditions: (i) $\dv>\cv>\bv>\av>\uv>\xv$, (ii)~restricted to the variables
$\uv$ the ordering corresponds to an orderly ranking (cf.\
Subsection~\ref{sec:diffideal}) and (iii) the variables $c^\alpha_\mu$ and
$b^\alpha_\mu$, respectively, are ordered among themselves in the same way
as the derivatives $u^\alpha_\mu$.

Let $S_{k,i}^{{\rm ex}}$ be one of the resulting simple algebraic
systems. If $S_{k,i}^{{\rm ex}}$ has less than
$\mathrm{codim} \Sol{(\I_{\ell,k})}$ equations with leader among the
auxiliary indeterminates $\cv$, $\dv$, we remove all equations with leader
among $\av$, $\bv$, $\cv$, $\dv$ and obtain the simple system $S_{k,i}$
over $\mathcal{D}_{\ell}$ which contributes $\Sol{(S_{k,i})}$ to
$\mathrm{ASing}\bigl(\overline{\Jc{\ell}(S)}\bigr)$. Otherwise, again
removing all equations with leader among $\av$, $\bv$, $\cv$, $\dv$, we
obtain a simple algebraic system $S_{k,i}$ in $\mathcal{D}_{\ell}$ which
contributes the regularity component $\Jcp{\ell,k}{i}=\Sol{(S_{k,i})}$. In
a more formal language, we arrive at
Algorithm~\ref{alg:simpleregularitydecomposition}.

\begin{algorithm}[Regularity Decomposition of a Simple Differential System] \quad
\label{alg:simpleregularitydecomposition}
\begin{algorithmic}[1]
\setlength{\baselineskip}{1.2\baselineskip}
\INPUT a simple differential system $S$ over $K\{ U \}$
and a sufficiently high order $\ell \in \NN$
\OUTPUT a regularity decomposition for each prime component
$\mathcal{I}_{\ell,k}(S)$ of the ideal $\mathcal{I}_\ell(S) \subset \mathcal{D}_{\ell}$
\ALGORITHM
\STATE compute a generating set $\{ p_1, \ldots, p_s \}$
of the radical ideal $\mathcal{I}_{\ell}(S)$ according to Remark~\ref{rmrk:genalgideal}
\STATE compute a prime decomposition $\mathcal{I}_{\ell}(S) = \mathcal{I}_{\ell,1}(S) \cap \ldots \cap \mathcal{I}_{\ell,t}(S)$ of $\mathcal{I}_{\ell}(S)$ and a generating set $\{ p_{k,1}, \ldots, p_{k,s_k} \}$ for each prime component $\mathcal{I}_{\ell,k}(S)$
\FOR{$k \in \{ 1, \ldots, t \}$}
\STATE compute an algebraic Thomas decomposition  $S_{k,1}^{\rm ex}$, \ldots, $S_{k,r_k}^{{\rm ex}}$ with respect to a total order $\dv >\cv >\bv > \av > \uv > \xv$ satisfying the above mentioned conditions of the
algebraic system
\begin{equation}\label{eq:determVessiot}
\left\{
\begin{array}{rcll}
\JJ(p_{k,j}) & = & 0\,, \\
\vv(p_{k,j}) & = & 0\,, \\ 
p_{k,j} & = & 0\,, 
\end{array}\right\}\qquad j = 1, \ldots, s_k
\end{equation}
defined over $\mathcal{D}_{\ell}^{{\rm \, ex}}$, where
\[
\vv = \sum_i a^i \, C_i^{(\ell)} + \sum_{|\mu| = \ell} \sum_{\alpha} b_{\mu}^{\alpha} \, C_{\alpha}^{\mu} \quad \mathrm{and} \quad 
\JJ= \sum_{\mu} \sum_{\alpha} c_{\mu}^{\alpha} \partial_{u^{\alpha}_{\mu}}  + \sum_i d^i \partial_{x^i}
\]
\ENDFOR
\STATE \textbf{return}
the systems $S_{k,i}$ consisting of those equations $p = 0$ and inequations $q \neq 0$ in $S_{k,i}^{{\rm ex}}$ with $p \in \mathcal{D}_{\ell}$ and $q \in \mathcal{D}_{\ell}$
\end{algorithmic}
\end{algorithm}

The remainder of this section is dedicated to explaining this algorithm and
proving its correctness.

\begin{remark}\label{rem:algothomas}
  Algorithm~\ref{alg:simpleregularitydecomposition} is in principle not yet
  completely specified, as we say nothing about how the algebraic Thomas
  decomposition in Step~4 is computed.  In fact, the correctness of the
  algorithm depends on whether this Thomas decomposition has been
  computed in a ``good'' way (this is made precise in
  Proposition~\ref{prop:disjointprojection}).  As any reasonable
  implementation automatically satisfies this condition, we have not
  mentioned it in the algorithm.  It also should be noted that the output
  of our algorithm depends not only on details of the implementation of the
  Thomas decomposition, but strongly on the used ranking.  In a system with
  several independent variables $\xv$ or several differential unknowns
  $\uv$, very different results can be obtained for different orderings
  inside each of the blocks $\xv$ and $\uv$, respectively.  In particular,
  the obtained regularity decomposition is often overly complicated,
  i.\,e.\ consist of too many different components, as in Step~4 we implicitly
  compute a Thomas decomposition of the variety $\Sol{\bigl(\I_{\ell,k}(S)\bigr)}$
  which might entail many case distinctions that are not necessary for our purposes.
  In a post processing step these unnecessary case distinctions could be conflated
  by comparing for
  all cases with smooth points the linear part corresponding to the system
  $\vv(p_{k,j})=0$ and combining all cases where equivalent equations have
  been obtained.
\end{remark}

In Remark~\ref{rmrk:genalgideal} it was already mentioned that the ideal
$\I_\ell(S)$ can be generated by a simple algebraic system.  Now
\cite[Thm.~1.94]{lh:phd} entails that this ideal is equidimensional in the
strong sense that all its associated primes have the same
dimension.\footnote{Some authors call such ideals unmixed dimensional and
  speak of equidimensional ideals already when all minimal primes have the
  same dimensions.}  In particular, no embedded primes can exist.  Because
of the assumptions that $S$ is a simple differential system and that $\ell$
is sufficiently large, the prime decomposition of $\I_\ell(S)$ computed in
the second step of our algorithm induces also a prime decomposition of the
differential ideal $\I_{\mathrm{diff}}(S)$, as we show below.  Note,
however, that even if the prime decomposition of $\I_\ell(S)$ is minimal,
there is no guarantee that also the differential prime decomposition is
minimal.  Here, we encounter again the well-known Ritt problem \cite[\textsection IV.9]{kol:diffalg} in
differential algebra: no algorithm is known to decide whether one
differential prime ideal is contained in another one.

For the next proof, it is important to discuss the relationship between the
notion of a simple differential system as defined in
Definition~\ref{def:diffsimple} and the notion of a regular differential
system often used in differential algebra -- see e.\,g.\
\cite[Def.~4.7]{eh:tri2}.  The following lemma and its proof entail that we
may always assume without loss of generality that a simple differential
system is also regular, as the only difference between these two notions is
the extent to which autoreduction has been performed.
\cite[Def.~4.7]{eh:tri2} uses partial reductions, i.\,e.\ only reductions
using derived equations are performed but no purely algebraic reductions.
However, it is always assumed that the whole differential polynomial is
reduced.  By contrast, the conditions imposed in
Definitions~\ref{def:algsimple} and \ref{def:diffsimple} require only head
reductions, but algebraic reductions are also performed.
From a theoretical point of view, it is irrelevant whether or not tail reductions are
performed.  From a computational point of view, they are often expensive
and thus it is better to omit them.

\begin{lemma}\label{lem:regsys}
  Let $S$ be a simple differential system as in
  \eqref{eq:differentialsystem}.  Then $S$ is equivalent to a regular
  differential system in the sense that some tail (pseudo) reductions turn
  $S$ into a regular system with the same leaders and the same saturated
  multiplicatively closed set generated by the initials and separants.
\end{lemma}
  
\begin{proof}
  From $S$ we collect the left hand sides of the equations and inequations,
  respectively, in the two sets $P$ and $Q$.  The first two
  properties in Definition~\ref{def:diffsimple} entail that modulo tail
  reduction $P$ is a differential triangular set.  The second property also
  ensures that all $\Delta$-polynomials that can be formed with elements of
  $P$ reduce to zero modulo $P$.  The first and third property imply that
  modulo tail reduction each inequation is partially reduced with respect
  to $P$.  Finally, denote by $Q^{\infty}$ the smallest subset of
  $\mathcal{D}$ that contains $1$ and $Q$ and has the property that
  $q,\tilde{q}\in Q^{\infty}$ is equivalent to $q\tilde{q}\in Q^{\infty}$,
  i.\,e.\ the saturated multiplicatively closed set generated by $Q$.  Note
  that tail (pseudo) reduction amongst elements of $P$ might change their
  initials and separants, by multiplying them by the initial or separants
  of the reducing polynomial.  Finally, any separant of an equation in $P$
  lies in $Q^{\infty}$, up to reduction by $P$.  Thus all conditions in the
  definition \cite[Def.~4.7]{eh:tri2} of a regular differential system are
  satisfied.
\end{proof}

\begin{proposition}\label{prop:diffprimedecomp}
  Assume that $S$ does not contain any equation or inequation of order
  greater than $\ell$ and denote by $q$ the product of all separants of the
  equations in $S$.  Then the differential ideals
  $\langle p_{k,1},\dots, p_{k,s_k}\rangle_\Delta:q^\infty$ (in the
  notation of Algorithm~\ref{alg:simpleregularitydecomposition}) for
  $1\leq k\leq t$ represent all essential prime components of the
  differential ideal~$\I_{\mathrm{diff}}(S)$.
\end{proposition}

\begin{proof}
  By Lemma~\ref{lem:regsys}, we may assume that $S$ is a regular
  differential system.  The statement follows then immediately from
  \cite[Thm.~4.13]{eh:tri2}.
\end{proof}

\begin{remark}
  When setting up the linear equations describing the Vessiot spaces in
  Step~4, it suffices to consider only those generators $p_{k,j}$ that
  depend on some jet variables of order $\ell$, as all other generators
  would only contribute the trivial equation $0=0$. Indeed, if $p$ is a
  generator of lower order, then we have trivially $C_\alpha^\mu(p)=0$ and
  it follows from \eqref{eq:fderiv} that $C_i^{(\ell)}(p)=D_ip$.  Since by
  Proposition~\ref{prop:diffprimedecomp} the ideal $\I_{\ell,k}(S)$ is the
  truncation of a differential ideal, the formal derivative~$D_ip$ (defined
  in Appendix~\ref{sec:geojet}) can be written as a linear combination of
  the generators~$p_{k,j}$. Hence $\mathbf{V}(p)$ vanishes modulo
  $\I_{\ell,k}(S)$, i.\,e.\ it is zero at all considered points.
\end{remark}

As preparation for showing the correctness of
Algorithm~\ref{alg:simpleregularitydecomposition}, we prove some results
about the simple algebraic systems $S_{k,i}^{{\rm ex}}$ produced by the
algorithm relating them to algebraic singularities and the Vessiot and
symbol spaces. Furthermore, we provide a technical proposition needed for
the correctness proof.

\begin{proposition}\label{prop:algSing}
  Given any subset
  $\Sol{(S_{k,i})}\subseteq V_{\ell,k}=\Sol{(\I_{\ell,k})}$, either all
  points contained in it are smooth in $V_{\ell,k}$ or all are algebraic
  singularities of $V_{\ell,k}$.
\end{proposition}

\begin{proof}
  By substituting the coordinates of a point $\rho \in V_{\ell,k}$ into
  \eqref{eq:jacobian}, we obtain a system of linear equations in $\cv$,
  $\dv$ whose solution space is the tangent space to $V_{\ell,k}$ at
  $\rho$. The point $\rho$ is smooth in $V_{\ell,k}$, if and only if this
  tangent space has dimension $\dim{V_{\ell,k}}$, and singular otherwise.
  The algebraic system (\ref{eq:determVessiot}) consists only of equations,
  and the equations which involve the indeterminates $\cv$, $\dv$ are
  homogeneous of degree one in these indeterminates. Since $\cv$, $\dv$ are
  ranked higher than the indeterminates $\uv$, $\xv$, we conclude that the
  simple algebraic system $S_{k,i}^{{\rm ex}}$ obtained in Step~4 of
  Algorithm~\ref{alg:simpleregularitydecomposition} contains no inequations
  with a leader among $\cv$, $\dv$ and every equation which involves the
  indeterminates $\cv$, $\dv$ is homogeneous of degree one in these
  indeterminates.  Consider now those equations with leader among $\cv$,
  $\dv$ in $S_{k,i}^{{\rm ex}}$. Due to the linearity and the triangularity
  of the system, the number of these equations is equal to the codimension
  of the tangent space and this codimension is independent of the choice of
  $\rho \in \Sol(S_{k,i})$, because $S_{k,i}^{{\rm ex}}$ is simple. Hence,
  $\Sol (S_{k,i})$ consists entirely of smooth points, if and only if the
  number of equations with leader among $\cv$, $\dv$ is equal to the
  codimension of $V_{\ell,k}$, and entirely of singular points otherwise.
\end{proof}

\begin{remark}\label{rem:algSing}
  No equation in the system \eqref{eq:determVessiot} contains
  simultaneously indeterminates from $\av$, $\bv$ and from $\cv$, $\dv$.
  Since the algebraic Thomas decomposition method does not apply polynomial
  division to a pair of equations involving different sets of
  indeterminates from $\av$, $\bv$ and $\cv$, $\dv$, respectively, the
  correctness does not depend on the choice of how $\av$, $\bv$, $\cv$,
  $\dv$ are ordered.
\end{remark}

\begin{proposition}\label{prop:dim}
  Let $S_{k,i}^{{\rm ex}}$ be a simple algebraic system obtained in Step~4
  of Algorithm~\ref{alg:simpleregularitydecomposition} such that
  $\Sol{(S_{k,i})}$ consists entirely of smooth points. Denote by $N_{\av}$
  and $N_{\bv}$ the number of equations with a leader among the variables
  $\av$ and $\bv$, respectively. Then at any point
  $\rho\in\Sol{(S_{k,i})}\subseteq\Sol{(\I_{\ell,k})}$ the dimension of the
  symbol space and of the Vessiot space, resp., is given by
  \begin{align}
    \dim{\V_{\rho}\bigl[\Sol{(\I_{\ell,k}(S))}\bigr]} &=
    m \binom{\ell + n - 1}{\ell} + n - N_{\bv} - N_{\av}\,,\\
    \dim{\N_{\rho}\bigl[\Sol{(\I_{\ell,k}(S))}\bigr]} &=
    m \binom{\ell + n - 1}{\ell}-N_{\bv}\,.
  \end{align}
  Furthermore, an $n$-dimensional complement $\H_{\rho}$ to the symbol
  space $\N_{\rho}\bigl[\Sol{(\I_{\ell,k}(S))}\bigr]$ exists in the Vessiot
  space $\V_{\rho}\bigl[\Sol{(\I_{\ell,k}(S))}\bigr]$, if and only if
  $N_{\av}=0$.  Finally, the set $\Sol{(S_{k,i})}$ is a regular algebraic
  jet set.
\end{proposition}

\begin{proof}
  The algebraic system (\ref{eq:determVessiot}) consists only of equations
  and the equations which involve the indeterminates $\av$ and $\bv$ are
  homogeneous of degree one in these indeterminates and independent of the
  indeterminates $\cv$ and $\dv$.  Thus, for notational simplicity, we may
  ignore in the sequel the equations containing $\cv$ and $\dv$.  Since
  $\av$ and $\bv$ are ranked higher than the indeterminates $\uv$, $\xv$,
  the simple algebraic system $S_{k,i}^{{\rm ex}}$ cannot contain
  inequations with a leader in $\av$ or $\bv$ and every equation which
  involves the indeterminates $\av$ and $\bv$ is still homogeneous of
  degree one in these indeterminates. The triangularity of
  $S_{k,i}^{{\rm ex}}$ means that these equations correspond to a reduced
  row echelon form of the determining equations of the Vessiot
  distribution. This row echelon form is preserved for any choice of the
  point $\rho \in {\rm Sol}(S_{k,i})$.  Since the dimensions of the vectors
  $\bv$ and $\av$ are $m \binom{\ell + n - 1}{\ell}$ and $n$, respectively,
  the claimed expression for the dimension of the Vessiot spaces follows
  immediately from the linearity of the equations.

  We consider next the symbol spaces. Let $\rho=(\uv, \xv)$ be a point on
  $\Sol{(S_{k,i})}$. The symbol space
  $\N_{\rho}\bigl[\Sol{(\I_{\ell,k}(S))}\bigr]$ consists of all solutions
  of $S_{k,i}^{{\rm ex}}$ of the form $(\bv, \vec{0}, \uv, \xv)$.  We rank
  $\bv$ higher than $\av$.  Hence any equation with a leader in $\av$ is
  independent of the indeterminates $\bv$ and can be ignored when the
  symbol is computed, as it is automatically satisfied by
  homogeneity.  This observation entails the claimed expression for the
  dimension of the symbol space.

  The dimension of any complement $\H_\rho$ is trivially the difference of
  the dimensions of the Vessiot and the symbol space.  Hence, by the just
  derived expressions for these dimensions, it is given by $n-N_{\av}$
  which proves the last assertion.  Finally, we note that by the above
  mentioned independence of the pivots in the row echelon form of the
  chosen point $\rho \in \Sol{(S_{k,i})}$, the dimensions of the Vessiot
  and the symbol spaces are constant over $\Sol{(S_{k,i})}$.  Hence this
  set is a regular algebraic jet set.
\end{proof}

One key point in proving the correctness of
Algorithm~\ref{alg:simpleregularitydecomposition} concerns the last step
when we move from the systems $S_{k,j}^{{\rm ex}}$ including the
indeterminates $\av$, $\bv$, $\cv$, $\dv$ to the projected systems
$S_{k,j}$.  The next proposition asserts that the disjointness is preserved
by this operation.  We consider the following generalisation of our
set-up. Let $R = \mathbb{C}[y_1, \ldots, y_m][z_1, \ldots, z_n]$ be a
polynomial ring equipped with the ranking
$z_1 > z_2 > \ldots > z_n > y_1 > y_2 > \ldots > y_m$. Let $S$ be a (not
necessarily simple) algebraic system over $R$ which does not contain any
inequation with a leader in $\{ z_1, \ldots, z_n \}$ and whose equations
with a leader in $\{ z_1, \ldots, z_n \}$ are homogeneous of degree one as
polynomials in $z_1$, \ldots, $z_n$. Applying any judicious algorithm
computing a Thomas decomposition -- e.\,g.\ the one from
\cite{bglr:thomasalg} -- to $S$ computes an output of this form (for the
necessity of this form see Remark~\ref{rem:compr} below), since both the
initials and the discriminants of the homogeneous polynomials of degree one
are polynomials in the variables $\mathbf{y}$ and hence no case distinction
with respect to any polynomial in the variables $\mathbf{z}$ is
necessary. Moreover, let $S_1$, \ldots, $S_r$ be an algebraic Thomas
decomposition of $S$ with respect to a ranking $>$ such that no $S_i$
contains an inequation with a leader in $\{ z_1, \ldots, z_n \}$. Our
situation is recovered by identifying the variables $\mathbf{z}$ with the
parameters $\av$, $\bv$, $\cv$, $\dv$ and the variables $\mathbf{y}$ with
the appearing jet variables. In the sequel, we denote by $(S_i)_{< z_j}$
the algebraic subsystem consisting of all equations and inequations in the
simple system $S_i$ with a leader less than $z_j$. Thus $(S_i)_{< z_n}$
corresponds to the projected system without any of the variables
$\mathbf{z}$.

\begin{proposition}\label{prop:disjointprojection}
  The solution sets ${\rm Sol}((S_1)_{< z_n})$, ${\rm Sol}((S_2)_{< z_n})$,
  \ldots, ${\rm Sol}((S_r)_{< z_n})$ of the projected systems are also
  pairwise disjoint.
\end{proposition}

\begin{proof}
  We first note that any subsystem $(S_i)_{< z_j}$ is also simple. By the
  properties of simple algebraic systems \cite[Subsect.~2.2.1]{dr:habil},
  every solution
\[
(\alpha_{j+1}, \alpha_{j+2}, \ldots, \alpha_n, \beta_1, \ldots, \beta_m) \in \mathbb{C}^{(n+m)-j}
\]
of the subsystem $(S_i)_{< z_j}$ can be extended to a solution
\[
(\alpha_j, \alpha_{j+1}, \alpha_{j+2}, \ldots, \alpha_n, \beta_1, \ldots, \beta_m) \in \mathbb{C}^{(n+m)-(j-1)}
\]
of the larger subsystem $(S_i)_{< z_{j-1}}$. Indeed, the subsystem $(S_i)_{< z_{j-1}}$ can differ from $(S_i)_{< z_j}$ by at most one additional equation or inequation with leader $z_{j-1}$ which then restricts the possible values for $\alpha_j$. In the case of $j=1$,
we set $(S_i)_{< z_0}:= S_i$. 

For $\beta = (\beta_1, \ldots, \beta_m) \in  \mathbb{C}^m$, we define the intersections
\[
\begin{array}{rcl}
V_{\beta} & := & {\rm Sol}(S) \cap \Bigl\{\, (z_1, \ldots, z_n, \beta_1, \ldots, \beta_m) \mid z_1, \ldots, z_n \in \mathbb{C}\, \Bigr\}
\end{array}
\]
and for $i \in \{ 1, \ldots, n \}$ and $(\alpha_{i+1}, \alpha_{i+2}, \ldots, \alpha_n) \in \mathbb{C}^{n-i}$ set
\begin{multline*}
  V_{\alpha,\beta} := \Bigl\{\, (\alpha_i, \alpha_{i+1}, \ldots,
  \alpha_n, \beta_1, \ldots, \beta_m) \mid
  \alpha_i \in \mathbb{C} \mbox{ and  }\\
  \exists\, \alpha_1, \ldots, \alpha_{i-1} \in \mathbb{C}\, :\,
  (\alpha_1, \ldots, \alpha_n, \beta_1, \ldots, \beta_m) \in {\rm Sol}(S)\,
  \Bigr\}.
\end{multline*}
Since the equations in $S$ with a leader in $\{ z_1, \ldots, z_n \}$ are homogeneous of degree one as polynomials in $z_1$, \ldots, $z_n$, for each $\beta = (\beta_1, \ldots, \beta_m) \in \mathbb{C}^m$ the set $V_{\beta}$ is either empty or an affine subspace of $\mathbb{C}^{n+m}$. For the same reason, for each $i \in \{ 1, \ldots, n \}$ and $\alpha = (\alpha_{i+1}, \alpha_{i+2}, \ldots, \alpha_n) \in \mathbb{C}^{n-i}$ the set $V_{\alpha,\beta}$ is also either empty or an affine subspace of $\mathbb{C}^{(n+m)-(i-1)}$.

Assume that ${\rm Sol}((S_{i_1})_{< z_n})$ and
${\rm Sol}((S_{i_2})_{< z_n})$ are not disjoint for $i_1 \neq i_2$. Let
$(\beta_1, \ldots, \beta_m) \in {\rm Sol}((S_{i_1})_{< z_n}) \cap {\rm
  Sol}((S_{i_2})_{< z_n})$. Since both $S_{i_1}$ and $S_{i_2}$ are simple
algebraic systems, the point $(\beta_1, \ldots, \beta_m)$ can be extended to
solutions
$\rho = (\alpha_1^{(i_1)}, \alpha_2^{(i_1)}, \ldots, \alpha_n^{(i_1)},
\beta_1, \ldots, \beta_m)$ and
$(\alpha_1^{(i_2)}, \alpha_2^{(i_2)}, \ldots, \alpha_n^{(i_2)}, \beta_1,
\ldots, \beta_m)$ of $S_{i_1}$ and $S_{i_2}$, respectively. The
disjointness of the solution sets ${\rm Sol}(S_{i_1})$ and
${\rm Sol}(S_{i_2})$ implies that there exists $k \in \{ 1, \ldots, n \}$
such that $\alpha_k^{(i_1)} \neq \alpha_k^{(i_2)}$. Let $k$ be maximal with
that property. Hence,
$(\alpha_k^{(i_1)}, \ldots, \alpha_n^{(i_1)}, \beta_1, \ldots, \beta_m)$
and
$(\alpha_k^{(i_2)}, \ldots, \alpha_n^{(i_2)}, \beta_1, \ldots, \beta_m)$
are two distinct elements of the affine subspace $V_{\alpha,\beta}$ of
$\mathbb{C}^{(n+m)-(k-1)}$, where
$\alpha = (\alpha_{k+1}^{(i_1)}, \alpha_{k+2}^{(i_1)}, \ldots,
\alpha_n^{(i_1)}) = (\alpha_{k+1}^{(i_2)}, \alpha_{k+2}^{(i_2)}, \ldots,
\alpha_n^{(i_2)})$. Therefore, $V_{\alpha,\beta}$ is not finite.

We introduce the index set
\[
I(\rho, k) = \Bigl\{\, i \in \{ 1, \ldots, r \} \mid (\alpha_{k+1}^{(i_1)}, \ldots, \alpha_n^{(i_1)}, \beta_1, \ldots, \beta_m) \in
{\rm Sol}((S_i)_{< z_k})\, \Bigr\}.
\]
Then we have $i_1$, $i_2 \in I(\rho, k)$ and
\[
V_{\alpha,\beta} = \bigcup_{i \in I(\rho, k)} {\rm Sol}((S_i)_{< z_{k-1}}).
\]
Since the affine subspace $V_{\alpha,\beta}$ of $\mathbb{C}^{(n+m)-(k-1)}$ is not finite, but $I(\rho, k)$ is finite, there exists $j_1 \in I(\rho, k)$ such that ${\rm Sol}((S_{j_1})_{< z_{k-1}})$ is infinite. Hence, $S_{j_1}$ contains no equation with a leader $z_k$. However, by assumption, $S_{j_1}$ contains no inequation with a leader $z_k$ either. By exchanging the roles of $i_1$ and $i_2$ if necessary, we may assume without loss of generality that $j_1 \neq i_1$. We conclude that $(S_{i_1})_{< z_{k-1}}$ and $(S_{j_1})_{< z_{k-1}}$ have the common solution $(\alpha_k^{(i_1)}, \ldots, \alpha_n^{(i_1)}, \beta_1, \ldots, \beta_m)$.

If $k = 1$, this observation contradicts the disjointness of
${\rm Sol}(S_{i_1})$ and ${\rm Sol}(S_{j_1})$. Otherwise, the thus obtained
common solution can be extended to a solution
$(\alpha_1^{(j_1)}, \alpha_2^{(j_1)}, \ldots, \alpha_{k-1}^{(j_1)},
\alpha_k^{(i_1)}, \ldots, \alpha_n^{(i_1)}, \beta_1, \ldots, \beta_m)$ of
$S_{j_1}$. By a similar reasoning as above, the disjointness of the
solution sets ${\rm Sol}(S_{i_1})$ and ${\rm Sol}(S_{j_1})$ implies that
there exists $l \in \{ 1, \ldots, k-1 \}$ such that
$\alpha_l^{(i_1)} \neq \alpha_l^{(j_1)}$. Let $l$ be maximal with that
property. Then $V_{\alpha',\beta}$ is not finite, where
$\alpha' = (\alpha_{l+1}^{(i_1)}, \alpha_{l+2}^{(i_1)}, \ldots,
\alpha_n^{(i_1)})$. Hence, there exists $j_2 \in I(\rho, l)$ such that
$S_{j_2}$ neither contains an equation with a leader $z_l$ nor an
inequation with a leader $z_l$. Without loss of generality, we may assume
$j_2 \neq i_1$. Then
$(\alpha_l^{(i_1)}, \ldots, \alpha_n^{(i_1)}, \beta_1, \ldots, \beta_m)$ is
a common solution of $(S_{i_1})_{< z_{l-1}}$ and $(S_{j_2})_{<
  z_{l-1}}$. If $l = 1$, this is a contradiction. Otherwise, this argument
can be repeated to obtain a contradiction. Hence, the sets
${\rm Sol}((S_1)_{< z_n})$, ${\rm Sol}((S_2)_{< z_n})$, \ldots,
${\rm Sol}((S_r)_{< z_n})$ are pairwise disjoint.
\end{proof}

\begin{remark}\label{rem:compr}
  The assumption in Proposition~\ref{prop:disjointprojection} about the
  absence of inequations with leader in $\{ z_1, \ldots, z_n \}$ cannot be
  omitted. For example, let $R = \CC[y][z_1, z_2]$ and $z_1 > z_2 > y$ and
  consider the system $S = \{ z_1 = 0 \}$. Then the systems
  \[
    S_1\colon \quad \left\{
      \begin{array}{rcl}
        z_1 & = & 0\\[0.2em]
        z_2 & = & 0
      \end{array}\right.
    \qquad
    S_2\colon \quad \left\{
      \begin{array}{rcl}
        z_1 & = & 0\\[0.2em]
        z_2 & \neq & 0\\[0.2em]
        y & = & 0
      \end{array}\right.
    \qquad
    S_3\colon \quad \left\{
      \begin{array}{rcl}
        z_1 & = & 0\\[0.2em]
        z_2 & \neq & 0\\[0.2em]
        y & \neq & 0
      \end{array}\right.
  \]
  provide a Thomas decomposition of $S$ with respect to $>$, where
  ${\rm Sol}((S_1)_{< z_2})$ and ${\rm Sol}((S_2)_{< z_2})$ are not
  disjoint and ${\rm Sol}((S_1)_{< z_2})$ and ${\rm Sol}((S_3)_{< z_2})$
  are not disjoint either. Note, however, that this Thomas decomposition
  involves case distinctions which would not occur in an application of the
  Thomas algorithm to $S$. In fact, the original algebraic system $S$ is
  already simple.  The assumption of
  Proposition~\ref{prop:disjointprojection} is automatically satisfied by
  any \emph{comprehensive} Thomas decomposition which also can be computed
  algorithmically \cite[Alg.~3.80]{baechler:phd}.
\end{remark}

\begin{theorem}\label{thm:algregdecomp}
  Algorithm~\ref{alg:simpleregularitydecomposition} terminates and is
  correct.
\end{theorem}

\begin{proof}
  The termination is obvious, as only terminating subalgorithms are
  used. For the correctness, it is sufficient to show that the output is
  correct for any prime component $\mathcal{I}_{\ell,k}(S)$ of
  $\mathcal{I}(S_{\ell})$. Let $k \in \{ 1, \ldots, t \}$.  We argue first
  that the output systems $S_{k,1}$, \ldots, $S_{k,r_k}$ form a Thomas
  decomposition. Since $S_{k,i}$ is obtained from the simple algebraic
  system $S_{k,i}^{{\rm ex}}$ by omitting the equations and inequations
  with a leader among $\dv$, $\cv$, $\bv$, $\av$, the algebraic system
  $S_{k,i}$ is simple.  In the proofs of Proposition~\ref{prop:algSing} and
  \ref{prop:dim}, it was shown that we are in a situation where
  Proposition~\ref{prop:disjointprojection} is applicable to the Thomas
  decomposition $S_{k,1}^{{\rm ex}}$, \ldots, $S_{k,r_k}^{{\rm
      ex}}$. Hence, the output systems $S_{k,1}$, \ldots, $S_{k,r_k}$ have
  pairwise disjoint solution sets ${\rm Sol}(S_{k,i})$ which either consist
  entirely of algebraic singularities by Proposition~\ref{prop:algSing} or
  are regular algebraic jet sets by Proposition~\ref{prop:dim}.
\end{proof}

Finally, we describe how one determines a regularity decomposition of a
general differential system $S$ in some order $\ell \in \NN$. The first
step is to compute a differential Thomas decomposition of $S$ into simple
differential systems $S_1$, $\dots$, $S_r$. For each simple system one
needs to check that the order $\ell$ we have chosen for the regularity
decomposition is sufficiently high. This means that we need to guarantee
that no equation or inequation in a simple differential system $S_i$ is cut
off when going from $S_i$ to the algebraic ideal $\I_\ell(S_i)$. If the
simple differential systems $S_1$, $\dots$, $S_r$ do contain an equation or
inequation of order greater than $\ell$, then a regularity decomposition in
this order is not possible. In this case one needs to adjust the order
$\ell$. If the order is high enough then one computes in a last step a
regularity decomposition of each simple differential system $S_i$ in order
$\ell$ with Algorithm~\ref{alg:simpleregularitydecomposition}. A formal
summary of this process is Algorithm~\ref{alg:regularitydecomposition}
below.

\begin{algorithm}[Regularity Decomposition of a General Differential System]\label{alg:regularitydecomposition} \quad
\begin{algorithmic}[1]
\setlength{\baselineskip}{1.2\baselineskip}
\INPUT a differential system $S$ defined over $K\{ u \}$ and order $\ell \in \NN$
\OUTPUT regularity decompositions in order $\ell$ of the irreducible components of the algebraic jet sets of the simple systems in a differential Thomas decomposition of $S$
\ALGORITHM
\STATE compute a differential Thomas decomposition $S_1$, \ldots, $S_r$ of the differential system $S$
\IF{one of the systems $S_i$ has an equation or an inequation of order greater than $\ell$}
\STATE \textbf{error}: order $\ell$ too small.
\ENDIF
\STATE \textbf{return} the regularity decompositions in order $\ell$ of the simple differential systems $S_i$
determined by Algorithm~\ref{alg:simpleregularitydecomposition}.
\end{algorithmic}
\end{algorithm}

\section{Regular Differential Equations}\label{sec:regde}

A basic assumption in most of the geometric theory of differential
equations is that one is dealing with a \emph{regular} equation.  This
means that not only the given differential equation
$\Jc{\ell}\subset\je{\ell}$ but also all its prolongations to higher order
are smooth manifolds on which symbol and Vessiot spaces define vector
bundles.  For nonlinear systems, it is generally very hard to verify these
infinitely many conditions and no effective method is known.  We now
provide a definition of regular differential equation adapted to our
framework and prove that we can identify in the output of
Algorithm~\ref{alg:simpleregularitydecomposition} one unique regular
equation for each irreducible component and that this equation lies dense
in the irreducible component.

The key problem encountered here is that the definition of a regular
differential equation requires to look at prolongations.  So far we could
avoid prolongations, as we assumed throughout that we start with a
differential system $S$ and then associate with it at any order $\ell$ an
algebraic jet set defined via the differential ideal $\I_{\mathrm{diff}}(S)$.
The problem of computing prolongations then corresponds to explicitly
constructing the polynomial ideals $\I_{\ell}(S)$, a question which has
been settled above.  By contrast, we assume now that we start with an
algebraic differential equation $\Jc{\ell}\subset\je{\ell}$ which is a
regular algebraic jet set in the sense of Definition~\ref{def:regjetset}.
The geometric theory describes an intrinsic prolongation process which,
however, assumes that one is dealing with a fibred submanifold.  In our
framework, this assumption is not necessarily satisfied and thus we must
develop another approach.

As a locally Zariski closed subset of $\je{\ell}$, we may consider
$\Jc{\ell}$ as the solution set of an algebraic system $S$ in the jet
variables up to order $\ell$.  Identifying the jet variables with the
derivatives of the dependent variables, we can also interpret $S$ as a
differential system which we associate with $\Jc{\ell}$.  Forming the
differential ideal $\I_{\mathrm{diff}}(S)$ corresponds to adding all
differential consequences of the equations describing $\Jc{\ell}$.
Obviously, this construction is independent of the choice of the algebraic
system $S$.

It may happen that $1\in\I_{\mathrm{diff}}(S)$.  In this case, the system
$S$ is differentially inconsistent and any further differential analysis is
pointless.  Otherwise, we consider for any $k\geq0$ the algebraic jet sets
$\Jc{\ell+k}(S)$. It may happen that $\Jc{\ell}(S)\subsetneq\Jc{\ell}$,
namely if some of the differential consequences are of an order less than
or equal to~$\ell$ (i.\,e.\ if hidden integrability conditions exist in
$S$).  In this case, it is again pointless to analyse $\Jc{\ell}$: one
should study $\Jc{\ell}(S)$ instead.  Otherwise, we call the algebraic jet
set $\Jc{\ell+k}=\Jc{\ell+k}(S)$ the $k$-\emph{th prolongation} of
$\Jc{\ell}$.

\begin{definition}\label{def:regde}
  The algebraic differential equation $\Jc{\ell}\subset\je{\ell}$ is called
  \emph{regular}, if the differential system $S$ associated with it
  satisfies
  \begin{enumerate}[label=\textup{(\roman*)},
                    itemsep=0pt,parsep=0pt,topsep=0pt]
  \item $\I_{\mathrm{diff}}(S)$ is a prime differential ideal,
  \item $\Jc{\ell}(S)=\Jc{\ell}$ and
  \item for all $k\geq0$ the algebraic jet sets $\Jc{\ell+k}(S)$ are
    regular and algebraic differential equations.
  \end{enumerate}
\end{definition}

Given an algebraic differential equation $\Jc{\ell}\subset\je{\ell}$, it is
not obvious how one can effectively verify that it is regular, since the
above definition comprises infinitely many condition as in the geometric
theory.  We now show that Algorithm~\ref{alg:simpleregularitydecomposition}
solves this problem to some extent,
as one can always identify in its output regular differential equations.

\begin{proposition}\label{prop:GenSimpSys}
  For each prime component $\I_{\ell,k}(S)$ arising in
  Algorithm~\ref{alg:simpleregularitydecomposition}, there exists among the
  simple systems $S_{k,i}$ in the output a unique distinguished
  system~$S_{\ell,k}^{\mathrm{gen}}$ such that
  $\Sol{(S_{\ell,k}^{\mathrm{gen}})}$ is Zariski dense in $\J_{\ell,k}$.
\end{proposition}

\begin{proof}
  System~\eqref{eq:determVessiot} comprises the equations
  $p_{k,1}=0,\dots , p_{k,s_k}=0$ defining the irreducible variety
  $\J_{\ell,k}$ and linear equations in the auxiliary indeterminates $\av$,
  $\bv$, $\cv$, $\dv$. Hence, the variety defined by
  \eqref{eq:determVessiot} is trivially fibred over $\J_{\ell,k}$ and
  therefore irreducible. By \cite[Cor.~2.2.66]{dr:habil}, any Thomas
  decomposition for an irreducible variety contains a unique simple system
  whose solution set is dense in that variety. Therefore there exists a
  unique index $i$ such that $\Sol (S_{k,i}^{{\rm ex}})$ is a dense subset
  of the variety defined by \eqref{eq:determVessiot}. Since
  $S_{k,i}^{{\rm ex}}$ contains no inequations with leader among the $\av$,
  $\bv$, $\cv$, $\dv$ and the equations involving $\av$, $\bv$, $\cv$,
  $\dv$ are homogeneous of degree one, the projected system $S_{k,i}$ has
  the claimed property.
\end{proof}

\begin{theorem}\label{thm:regde}
  In the notation of Proposition~\ref{prop:GenSimpSys}, the set
  $\Sol{(S_{\ell,k}^{\mathrm{gen}})}$ is a regular differential equation.
\end{theorem}

\begin{proof}
  Assume for notational simplicity that already $\I_{\mathrm{diff}}(S)$ is
  a prime differential ideal so that we can drop the index $k$. In this
  case, the ideal $\I_{\ell}(S)$ is generated by the triangular set
  $B_{\leq\ell}$ defined in \eqref{eq:Bl} followed by a saturation with
  respect to the inequations in $S$ (cf.\ Equations~\eqref{eq:dnss} and
  \eqref{eq:dnsssimp}). Since our ordering of the variables $\mathbf{c}$
  and $\mathbf{b}$, respectively, is linked to an orderly ranking of the
  derivatives $\mathbf{u}$, the two linear subsystems of
  \eqref{eq:determVessiot} arise now immediately in a row echelon
  form\footnote{Recall from Remark~\ref{rem:sat} that the saturation only
    eliminates unwanted points.  Hence at the remaining points we can use
    for the construction of the tangent space the equations in the
    triangular set $B_{\leq\ell}$ instead of some ideal generators obtained
    after the saturation.} and its pivots are separants of the equations in
  $B_{\leq\ell}$.  Furthermore, in the generic system
  $S_{\ell}^{\mathrm{gen}}$ of the algebraic Thomas decomposition the
  separants and initials of the equations in $B_{\leq\ell}$ are implied
  being non-zero. It is now trivial to see that
  $\Sol{(S_{\ell,k}^{\mathrm{gen}})}$ is a regular algebraic jet set.

  We now show that the same holds for
  $\Sol{(S_{\ell+1}^{\mathrm{gen}})}$, the generic branch obtained by
  applying our algorithm at the next order. By induction, we obtain that
  the generic branch defines at any prolongation order a regular algebraic
  jet set and thus our claim. By the same arguments as above, the ideal
  $\I_{\ell+1}(S)$ is generated by the triangular set $B_{\leq\ell+1}$
  followed by a saturation.  Since we assume that $S$ contains no equations
  or inequations of an order greater than $\ell$, $B_{\leq\ell+1}$ is
  obtained by augmenting $B_{\leq\ell}$ by certain formal derivatives
  (defined in Appendix~\ref{sec:geojet}) of its elements of order
  $\ell$. By the properties of the formal derivative, the new elements are
  linear in their leaders and their initials (and thus also their
  separants) are the separants of the elements of $B_{\leq\ell}$ from which
  they are derived. This implies that no new separants or initials arise
  during the prolongation. Again, these separants and initials are implied
  to be non-zero by the algebraic Thomas decomposition. Since again the
  linear subsystems of \eqref{eq:determVessiot} arise immediately in
  triangular form with separants as pivots, the made observation about the
  separants entails trivially that $\Sol{(S_{\ell+1}^{\mathrm{gen}})}$ is a
  regular algebraic jet set, too.

  For the general case, we exploit again that, by Lemma~\ref{lem:regsys},
  we may assume that $S$ is a regular differential system.
  \cite[Thm.~4.13]{eh:tri2} asserts that any characteristic set $C$
  describing a prime component of $\I_{\mathrm{diff}}(S)$ has the same
  leaders as the differential system $S$.  In
  Algorithm~\ref{alg:simpleregularitydecomposition}, we first compute in
  Step~2 some basis for each prime component $\I_{\ell,k}(S)$ and then in
  Step~4 perform an algebraic Thomas decomposition.  The generic branch of
  this decomposition determines a characteristic set~$C_{\ell,k}$
  describing $\I_{\ell,k}(S)$, namely the equations in
  $S_{\ell,k}^{\mathrm{gen}}$.  Furthermore, among the inequations in
  $S_{\ell,k}^{\mathrm{gen}}$ we must find the initials and separants of
  $C$.  As in the proof of Proposition~\ref{prop:diffprimedecomp},
  \cite[Thm.~4.13]{eh:tri2} allows us to interpret $C$ also as a
  differential characteristic set.  By definition of a simple differential
  system, $S$ is passive for the Janet division.  This implies that $C$
  must also be passive for the Janet division.  Indeed, otherwise $C$ would
  induce integrability conditions and any characteristic set of the ideal
  induced by $C$ would require additional leaders which contradicts
  \cite[Thm.~4.13]{eh:tri2}.  But now we can apply to $C$ exactly the same
  reasoning as in the special case above and conclude that
  $\Sol{(S_{\ell,k}^{\mathrm{gen}})}$ is a regular differential equation.
\end{proof}

\begin{corollary}\label{cor:regde}
  For any index $k$, the set $\Sol{(S_{\ell,k}^{\mathrm{gen}})}$ consists
  entirely of regular points of the algebraic differential equation
  $\J_{\ell,k}$.
\end{corollary}

\begin{proof}
  By the considerations in the proof of Theorem~\ref{thm:regde}, the
  equations in $S_{\ell,k}^{\mathrm{gen}}$ are passive for the Janet
  division.  Since $S_{\ell,k}^{\mathrm{gen}}$ arises from an algebraic
  Thomas decomposition, it is a simple algebraic system.  No leader of an
  inequation is the derivative of the leader of an equation, as all
  (suitable, cf.\ Equation~\eqref{eq:Bl}) derivatives of the differential
  equations have been added as algebraic equations.  Hence,
  $S_{\ell,k}^{\mathrm{gen}}$ is also simple as a differential system.  It
  follows now from Theorem~\ref{thm:regdense} that the regular points form
  a Zariski dense subset of $\J_{\ell,k}$.  Since
  $\Sol{(S_{\ell,k}^{\mathrm{gen}})}$ is also Zariski dense in
  $\J_{\ell,k}$ by Proposition~\ref{prop:GenSimpSys}, it contains regular
  points.  By Proposition~\ref{prop:dim}, this means that at all of its points
  the Vessiot and symbol spaces have the right dimensions.  Furthermore, we
  have seen above that at the points in $\Sol{(S_{\ell,k}^{\mathrm{gen}})}$
  no initial or separant vanishes.  Hence, we can conclude as in the proof
  of Theorem~\ref{thm:regdense} that $\Sol{(S_{\ell,k}^{\mathrm{gen}})}$ is
  actually an involutive differential equation and thus that around each
  point the required involutive complement to the symbol spaces exists.
\end{proof}

\section{Examples}\label{sec:exmp}

\begin{example}\label{exmp:finite_pde_saturation}
  We continue Example~\ref{ex:sat}. There it was already mentioned that a
  differential Thomas decomposition of the differential system defined by
  the partial differential equations $p_1=0$ and $p_2=0$ with $p_1$ and
  $p_2$ given by \eqref{eq:satex} yields only one simple differential
  system comprising besides the two given equations the inequation
  $\spt{(p_1)}=u\neq0$. Now we want to apply
  Algorithm~\ref{alg:simpleregularitydecomposition} for the determination
  of the geometric singularities of this simple differential system in
  order $\ell=1$, or more precisely a regularity decomposition of
  $\Jc{1}(S)$. All different types of singularities introduced in
  Definition~\ref{def:regsing} appear in this example.

  The first step of Algorithm~\ref{alg:simpleregularitydecomposition}
  requires the saturation already discussed in Example~\ref{ex:sat} which
  leads to the addition of a third generator $p_3$ given by \eqref{eq:satp3}.
  The algebraic ideal $\I_1(S)$ generated by these three generators is prime.
  second step. It was already mentioned above that $p_1$ is a linear
  combination of $p_2$ and $p_3$ and thus can in principle be omitted.  As
  the equations $p_2=0$ and $p_3=0$ can be solved for $u$ and $y$,
  respectively, the variety $\J_1(S)$ is a graph and thus no algebraic
  singularities occur here. Therefore we ignore in the sequel the
  equations $\mathbf{J}(p_k)=0$. In general, such a redundancy is not easy
  to recognise and therefore we do not exploit it any more in the
  following computations. The linear part of the system
  \eqref{eq:determVessiot} defining the Vessiot spaces takes here the form
  \begin{align*}
    \begin{pmatrix}
      u & 0 & u_x(u_x-y) & -2y-u+u_y(u_x-y) \\
      0 & y & -u_x & 0 \\
      u_y & u_x & -u_x & -1-u_y
    \end{pmatrix}
    \cdot
    \begin{pmatrix}
      b_{10}\\b_{01}\\a^1\\a^2
    \end{pmatrix}
    =0\,.
  \end{align*}
  The nonlinear part is given by $p_1=p_2=p_3=0$. The algebraic Thomas
  decomposition of this system performed in Step~4 yields
  after the projection in the last step the following four systems
  \begin{align*}
    S_1 &:= \bigl\{p_1 = 0, p_2 = 0, u\not=0, y\not=0\bigr\}\,,\\
    S_2 &:= \bigl\{u_x = 0, u_y \not= 0, u = 0, y = 0\bigr\}\,,\\
    S_3 &:= \bigl\{u_x \not= 0, u_y = 0, u = 0, y = 0\bigr\}\,,\\
    S_4 &:= \bigl\{u_x = 0, u_y = 0, u = 0, y = 0\bigr\}\,.
  \end{align*}
  We now show that the corresponding algebraic jet sets $\Jc{1}(S_i)$ are
  all regular and thus define a regularity decomposition of our system in
  order $1$. Obviously, $\Jc{1}(S_1)$ is a Zariski open subset of a
  three-dimensional variety in $\je{1}$. $\Jc{1}(S_2)$ and $\Jc{1}(S_3)$
  are disjoint Zariski open subsets of two-dimensional varieties lying in
  the Zariski closure of $\Jc{1}(S_1)$. Finally, $\Jc{1}(S_4)$ is a curve
  lying in the intersection of the Zariski closures of all the other
  systems. Of the four jet sets, only $\Jc{1}(S_1)$ is an algebraic
  differential equation, as for the other three systems the projections
  $\pi^1\bigl(\Jc{1}(S_i)\bigr)$ violate the closure condition of
  Definition~\ref{def:ade} because of the equation $y=0$.

  We finally discuss the Vessiot spaces for the points on these algebraic
  jet sets so that we can classify them according to the taxonomy of
  Definition~\ref{def:regsing}.  The Vessiot spaces are determined by the
  solutions of those (homogeneous linear) equations in the algebraic
  systems obtained after Step~4 that depend on $\mathbf{a}$ and
  $\mathbf{b}$. We describe them in terms of their coefficient matrices.
  For points on $\Jc{1}(S_1)$ we have the matrix
  \begin{align*}
    \begin{pmatrix}
      u^3 & 0 & y^3(u+y) & -u^2(y+u) \\
      0 & u & -u-y & 0
    \end{pmatrix}\,.
  \end{align*}
  The two corresponding equations express $b_{10}$ and $b_{01}$ in terms of
  the unconstrained variables $a^1$ and $a^2$.  Thus all Vessiot spaces are
  two-dimensional and all symbol spaces vanish so that the Vessiot spaces
  are transversal. Hence, all the points on the algebraic jet set
  $\Jc{1}(S_1)$ are regular points of the differential equation
  $\Jc{1}(S)$.  By Theorem~\ref{thm:regde}, $\Jc{1}(S_1)$ is a regular
  differential equation, as obviously $S_{1}$ is the generic branch in the
  algebraic Thomas decomposition.  Thus our findings are consistent with
  Corollary~\ref{cor:regde}.

  An analogous comparison of the dimensions of Vessiot and symbol spaces,
  respectively, determines the singular character of the points on
  $\Jc{1}(S_2)$, $\Jc{1}(S_3)$ and $\Jc{1}(S_4)$.  The respective Vessiot
  spaces are the kernels of the three matrices:
  \begin{align*}
    \begin{pmatrix}
      u_y & 0 & 0 & -1-u_y
    \end{pmatrix}\,,\qquad
    \begin{pmatrix}
      0 & u_x & 0 & -1 \\
      0 & 0 & 1 & 0
    \end{pmatrix}\,,\qquad
    \begin{pmatrix}
      0 & 0 & 0 & 1
    \end{pmatrix}\,.
  \end{align*}
  All points on $\Jc{1}(S_2)$ are purely irregular singular, as their
  Vessiot spaces are three-dimensional, but still contain a two-dimensional
  transversal part.  At points on $\Jc{1}(S_3)$, the dimension of the
  Vessiot spaces is still two; however, the transversal part is only
  one-dimensional.  Hence, they are regular singular.  Finally, the Vessiot
  spaces at points on $\Jc{1}(S_4)$ are three-dimensional with only
  one-dimensional transversal complements to the symbol spaces.  Thus,
  these points are irregular singular.  These considerations also prove
  that all sets $\Jc{1}(S_i)$ are regular algebraic jet sets and hence the
  four sets together define a regularity decomposition of $\Jc{1}(S)$ in
  order $1$.
\end{example}

\begin{example}\label{ex:hypgather}
  The hyperbolic gather is a classical example from catastrophe theory and
  is defined by the differential polynomial $p:=(u')^{3}+uu'-x$.  A real
  picture corresponding algebraic differential equation is given by the
  blue surface on the left in Figure~\ref{fig:hypgathintro} with its fold
  line shown in red.  All singularities lie on this fold line.  On the
  right, Figure~\ref{fig:hypgathintro} shows some (real) solution graphs
  and one can see how solutions reach a forward or backward impasse when
  they hit the projection of the fold line shown in black.

  \begin{figure}
    \centering
    \includegraphics[height=4.5cm]{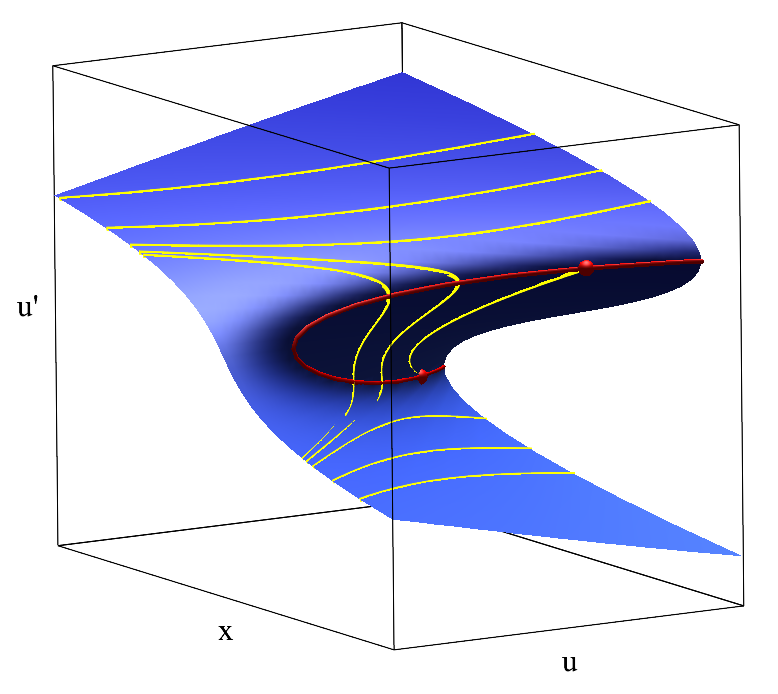}\quad
    \includegraphics[height=4.5cm]{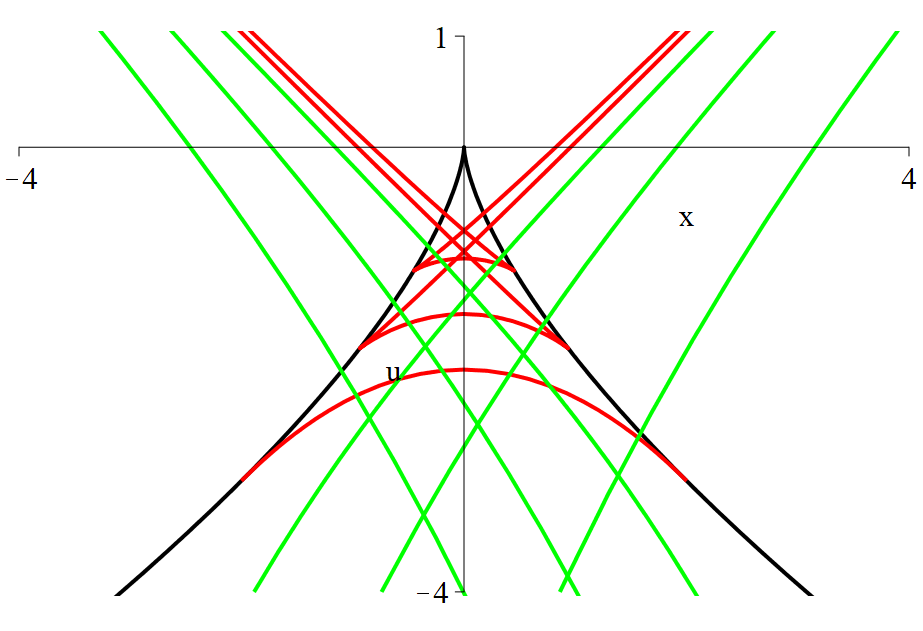}
    \caption{Hyperbolic gather. Left: Surface with singularities in jet
      space.  Right: solution graphs---note how the red curves ``go
      backwards'' after meeting the black curve, a generic behaviour at
      regular singularities.}
    \label{fig:hypgathintro}
  \end{figure}

  The hyperbolic gather represents probably one of the simplest examples to
  demonstrate the artifacts that the algebraic Thomas decomposition may
  introduce in the output of
  Algorithm~\ref{alg:simpleregularitydecomposition} (compare
  Remark~\ref{rem:algothomas}).  Using the implementation presented in
  \cite{bglr:thomasalg}, one obtains a regularity decomposition consisting
  of seven components (all composed of smooth points).  One of them
  consists of the two irregular singularities shown as distinguished points
  on the fold line in Figure~\ref{fig:hypgathintro}; three other components
  describe the remainder of the fold line (one of them singles out the
  ``tip'' of the fold line, one contains only complex points not visible in
  the real picture).

  The remaining three components contain the regular points.  The
  corresponding extended algebraic systems are given by
  \begin{align*}
    S_{1}^{\mathrm{ex}} &:= \bigl\{p=0,\ 4u^{3}+27x^{2}\neq0,\ x\neq0,\
            (3(u')^{2}+u)b+((u')^{2}-1)a=0\bigr\}\,, \\
    S_{2}^{\mathrm{ex}} &:= \bigl\{(u')^{3}+uu'=0,\ x=0,\ u\neq0,\
            (3(u')^{2}+u)b+((u')^{2}-1)a=0\bigr\}\,, \\
    S_{3}^{\mathrm{ex}} &:= \bigl\{uu'+3x=0,\ 4u^{3}+27x^{2}=0,\ x\neq0,\
            81x^{2}b+(36x^{2}-4u^{2})a=0\bigr\}\,,
  \end{align*}
  where we omitted the equations corresponding to the Jacobian criterion.
  Here $S_{1}$ is obviously the unique distinguished system of
  Proposition~\ref{prop:GenSimpSys} defining a regular differential
  equation.  However, if one takes the respective equations for the Vessiot
  space into account, then one sees\footnote{In the case of $S_{3}$, this
    requires that one takes the coefficients as they appear in $S_{1}$ and
    $S_{2}$ and rewrites them modulo the equations in $S_{3}$.} that
  the distinction between the three systems has no meaning for our analysis
  of singularities.

  \begin{figure}
    \centering
    \includegraphics[height=5cm]{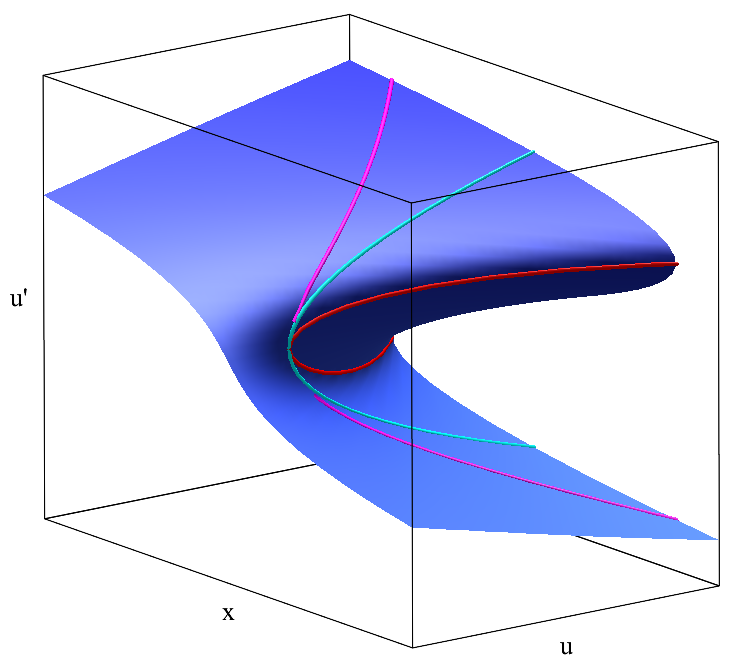}
    \caption{Hyperbolic gather with redundant case distinctions.}
    \label{fig:hypgather}
  \end{figure}

  The appearance of these unnecessary case distinctions can be easily
  explained from the geometry of the corresponding algebraic jet set
  $\J_{1}$ shown once more in Figure~\ref{fig:hypgather} over the real
  numbers.  Again, the red curve shows all geometric singularities of
  $\J_{1}$.  The set $\Sol{S_{3}}$ consists of those points of $\J_{1}$
  which lie either above or below the fold line and is shown in magenta.
  This set must be singled out by any algebraic Thomas decomposition for
  the ordering $u'>u>x$, as at its projection to the $x$-$u$ plane the
  fibre cardinality changes (this statement remains true over the complex
  numbers, as the hyperbolic gather simply depicts the solutions of the
  reduced cubic equation in $u'$ with coefficients $u$ and $-x$).  Finally,
  the set $\Sol{S_{2}}$ -- shown in cyan -- contains those points where the
  discriminant of the discriminant of $p$, i.\,e.\ $x$, vanishes.  This set
  has a geometric relevance only at its intersection with the fold line, as
  it singles out the point where the fold line itself folds (respectively
  where the underlying cubic equation has a triple zero).  Because of the
  inner working of the algorithm used to compute an algebraic Thomas
  decomposition, this condition leads to a separate case.
\end{example}

\begin{example}\label{ex:sphereplane}
  We consider now a situation where one is dealing with a reducible variety
  so that the second step of
  Algorithm~\ref{alg:simpleregularitydecomposition} becomes
  non-trivial. Its treatment demonstrates why we prefer to consider
  only irreducible varieties. The starting point is the differential system
  consisting of only one equation in factored form,
  \begin{equation}
    p:=(u'-c)\bigl((u')^2+u^2+x^2-1\bigr)=0\,,
  \end{equation}
  where $c\in[-1,1]$ is a real constant and no inequation. A differential
  Thomas decomposition yields a single simple differential system $S$ which
  contains besides the equation $p=0$ only the inequation $\spt{(p)}\neq0$.

  Algorithm~\ref{alg:simpleregularitydecomposition} (for $\ell=1$) computes
  in the first step the algebraic ideal $\I_1(S)$ which is here simply
  generated by $p$, as the saturation has no effect. Its prime
  decomposition yields two prime ideals generated by the two factors of
  $p$: $p_1=u'-c$ and $p_2=(u')^2+u^2+x^2-1$. Considered over the reals, we
  are dealing here with a sphere and a horizontal plane intersecting
  it. Obviously, both irreducible varieties are without algebraic
  singularities so that we ignore in the sequel the equations
  $\JJ(p_k)=0$. It is trivial to see that a regularity decomposition
  of $\Jc{1}(p_1)$ yields only one regularity component, namely
  $\Jc{1}(p_1)$ itself, and all points on it are regular.  In particular,
  $\Jc{1}(p_1)$ is trivially a regular differential equation.

  The second irreducible component was already analysed in
  \cite[Ex.~9.1.12]{wms:invol}. The linear equation for the Vessiot spaces
  is $2 u' b + (2 u u' + 2 x) a = 0$. For the ranking $b > a > u' > u > x$,
  the implementation presented in \cite{bglr:thomasalg} determines an
  algebraic Thomas decomposition consisting of five simple algebraic
  systems (since no algebraic singularities exist on this component, we
  ignore again the part stemming from the Jacobian criterion):
\[
\begin{array}{cc}
   S_1^{\mathrm{ex}}\colon
\left\{
\begin{array}{rcl}
u'b+(uu'+x)a & = & 0,\\[0.2em]
(u')^2+u^2+x^2-1 & = & 0,\\[0.2em]
u^2+x^2-1 & \neq & 0,\\[0.2em]
x^2-1 & \neq & 0
\end{array}\right.\qquad  & 
S_2^{\mathrm{ex}}\colon
\left\{
\begin{array}{rcl}
u'b+(uu'+x)a & = & 0,\\[0.2em]
(u')^2+u^2 & = & 0,\\[0.2em]
u & \neq & 0,\\[0.2em]
x^2-1 & = & 0
\end{array}\right.
\end{array}
\]
\[
\begin{array}{ccc}
S_3^{\mathrm{ex}}\colon
\left\{
\begin{array}{rcl}
a & = & 0,\\[0.2em]
u' & = & 0,\\[0.2em]
u^2+x^2-1 & = & 0,\\[0.2em]
x^3-x & \neq & 0
\end{array}\right. & 
S_4^{\mathrm{ex}}\colon
\left\{
\begin{array}{rcl}
a & = & 0,\\[0.2em]
u' & = & 0,\\[0.2em]
u & = & 0,\\[0.2em]
x^2-1 & = & 0
\end{array}\right.  & 
S_5^{\mathrm{ex}}\colon
\left\{
\begin{array}{rcl}
u' & = & 0,\\[0.2em]
u^2-1 & = & 0,\\[0.2em]
x & = & 0
\end{array}\right.
\end{array}
\]

The reduced systems corresponding to the first two systems
$S_1^{\mathrm{ex}}$ and $S_2^{\mathrm{ex}}$ can be combined into one simple
algebraic system leading to the following subset of the differential
equation  $\Jc{1}(p_2)$:
\begin{equation}\label{eq:r1}
  \Rc{1} = {\rm Sol}(S_1) \cup {\rm Sol}(S_2) =
  {\rm Sol}\Bigl(\bigl\{ p_2=0,\ u' \neq 0 \bigr\}\Bigr)\,.    
\end{equation}
Such a combination is also possible for the third and the fourth system and
yields another subset of $\Jc{1}(p_2)$ disjoint of $\Rc{1}$:
\begin{equation}\label{eq:r2}
  \Rc{2} = {\rm Sol}(S_3) \cup {\rm Sol}(S_4) =
  {\rm Sol}\Bigl(\bigl\{ p_2=0,\ u' = 0, u^2 - 1 \neq 0,\ x \neq 0 \bigr\}\Bigr)\,.
\end{equation}
We have thus constructed a regularity decomposition of $\Jc{1}(p_2)$ with
three regularity components: $\Rc{1}$ and $\Rc{2}$ as defined above and
$\Rc{3}={\rm Sol}(S_5)$.

We now classify the points on these three components according to the
taxonomy of Definition~\ref{def:regsing}. By Proposition~\ref{prop:dim}, we
have $\dim{\V_{\rho}[\Jc{1}(p_2)]} = 1$ for all points $\rho \in
\Rc{1}$. Moreover, for these points we have $u' \neq 0$. Since $u'$ is the
initial of the equation with leader $b$, the assumption $a = 0$ implies
$b = 0$ and hence the symbol space $\N_\rho[\Jc{1}(p_2)]$ is trivial. We
conclude that all points in $\Rc{1}$ are regular.  It follows again from
Proposition~\ref{prop:dim} that $\dim{\V_{\rho}[\Jc{1}(p_2)]} = 1$ also for
all points $\rho \in \Rc{2}$. Since the condition $a=0$ belongs to the
equations describing $\Rc{2}$, all these Vessiot spaces are vertical,
i.\,e.\ $\V_{\rho}[\Jc{1}(p_2)]=\N_\rho[\Jc{1}(p_2)]$ everywhere on
$\Rc{2}$. Thus all these points are regular singular. As the system $S_5$
defining $\Rc{3}$ contains no equations depending on $a$ or $b$, everywhere
on $\Rc{3}$ the Vessiot spaces are two-dimensional and hence all points
there are irregular singular.

In this example, it is not difficult to verify that $\Rc{1}$ is a regular
differential equation, although Theorem~\ref{thm:regde} guarantees this
only for the dense subset $\Sol{S_{1}}$. The inequation $x^{2}-1\neq0$ is
irrelevant for the initials and separants of $S_{1}$ and the systems
$S_{1}^{\mathrm{ex}}$ and $S_{2}^{\mathrm{ex}}$ contain exactly the same
equation for the coefficients $a$ and $b$ in our ansatz for the Vessiot
space.

We can now compare the results for $\Jc{1}(p_1)$ and $\Jc{1}(p_2)$ for the
points on their intersection, i.\,e.\ at points which are algebraic
singularities of the original reducible variety $\Jc{1}(S)$. If $c\neq0$,
then the points on $\Jc{1}(p_1)\cap\Jc{1}(p_2)$ have been classified as
regular for both irreducible components. However, for $c=0$ the points on
the intersection are still regular with respect to $\Jc{1}(p_1)$, but
regular singular with respect to $\Jc{1}(p_2)$. This exemplifies again the
statement made in the beginning of Example~\ref{exmp:clairaut}
that the taxonomy of Definition~\ref{def:regsing} is
relative and strongly depends on the considered algebraic jet set.

A natural question in such a situation is whether generalised solutions
exist which lie on both components. Let us assume for simplicity that
$c\neq0$.  Then on each component there exists a unique generalised
solution going through $\rho$.  Over the complex numbers, the identity
theorem for holomorphic functions excludes the possibility to combine
pieces of these to new solutions.  Over the real numbers, solutions of
lower regularity are admitted even if we restrict to classical solutions.
In our case, we can construct additional solutions through $\rho$ by
approaching $\rho$ on one of these two solutions and by then ``switching''
to the other one.  As the resulting curve in $J_{1}\pi$ is still
continuous, it corresponds to the prolongation of a function which is at
least $\mathcal{C}^{1}$ at the value~$x$ where the switching occurs.

As the direction of the tangent of a generalised solution encodes the value
of the second derivative, a necessary and sufficient condition for the thus
constructed solution to be even $\mathcal{C}^{2}$ at $x$ is that at the
intersection point the Vessiot spaces with respect to the two irreducible
components are identical. In our case, all Vessiot spaces
$\V_{\rho}[\Jc{1}(p_1)]$ are spanned by the vector
$\partial_x+c\partial_u$, whereas a basis of the Vessiot space
$\V_{\rho}[\Jc{1}(p_2)]$ at any point
$\rho=(\bar{x},\bar{u},\bar{p})\notin\Rc{3}$ is given by the vector
$\bar{p}(\partial_x+\bar{p}\partial_u)-(\bar{x}+\bar{u}\bar{p})\partial_p$. If
we assume that we are on the intersection, i.\,e.\ that $\bar{p}=c$ and
$\bar{u}^2+\bar{x}^2=1-c^2$, then it is easy to see that the Vessiot spaces
can be identical only for $c\neq0$ and then this happens only at the
two points
\begin{displaymath}
  \rho_{\pm}=\biggl( \mp c\sqrt{\frac{1-c^2}{1+c^2}},
  \pm \sqrt{\frac{1-c^2}{1+c^2}},c\biggr)\,.
\end{displaymath}
By analysing the next prolongation of our equation, it is not difficult to
show that the ``switching'' solutions are exactly $\mathcal{C}^{2}$, as the
value of the second derivative jumps at the switching point.

\begin{figure}
  \centering
  \includegraphics[height=5cm]{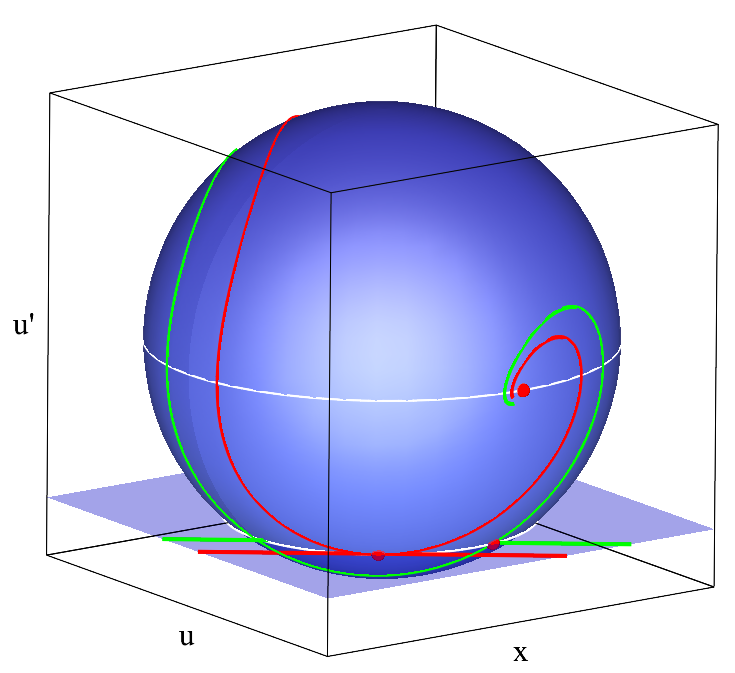}\quad
  \includegraphics[height=5cm]{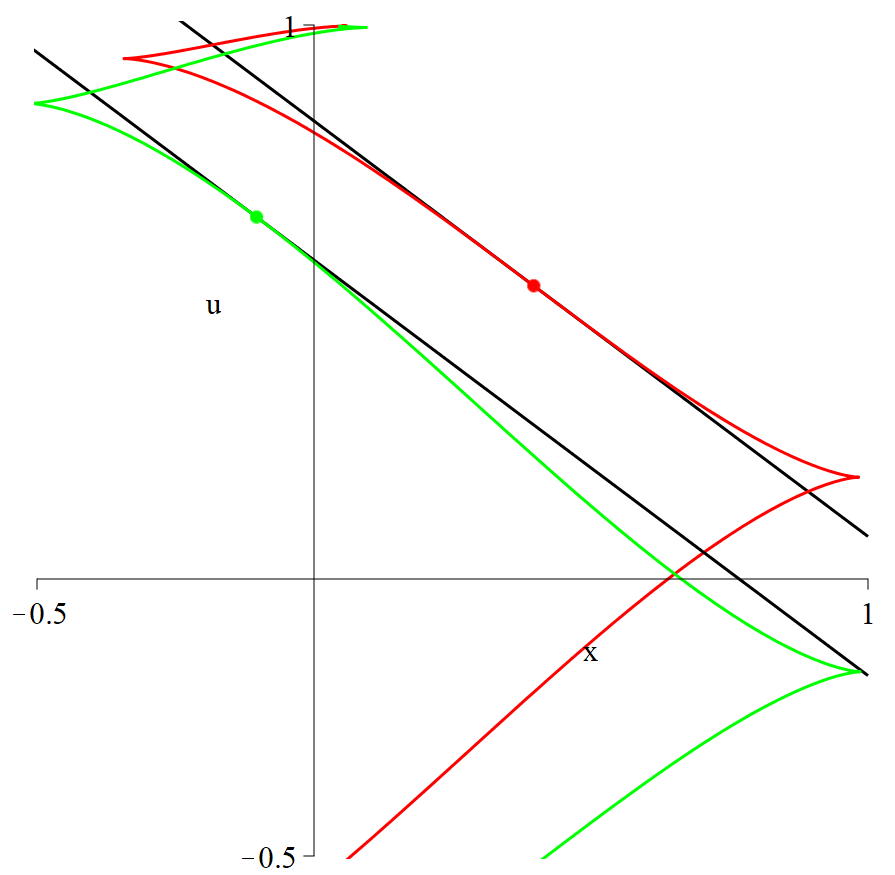}
  \caption{First-order differential equation with two irreducible
    components. Left: generalised solutions in $J_{1}\pi$. Right: 
     solution graphs in $x$-$u$ plane.}
  \label{fig:sphereplane}
\end{figure}

Thus we can conclude that over the real numbers we have through each point
on the intersection four solutions: two analytic functions with
prolongations staying completely on one component and two $\mathcal{C}^{1}$
functions switching between components.  For the points $\rho_{\pm}$ the
latter two solutions are even $\mathcal{C}^{2}$; a higher regularity is not
possible for ``switching'' solutions.  Figure~\ref{fig:sphereplane}
provides a graphical presentation of the situation over the reals for the
choice $c=-\tfrac{3}{4}$.  The red curves intersect at $\rho_{+}$; the
green curves at some point different from $\rho_{\pm}$.  Geometrically,
$\rho_{\pm}$ are distinguished by the fact that the value of $u'$ at these
points represents a local extremum along the generalised solution of
$\mathcal{J}_{1}(p_{2})$ going through it.  This in turn means that the
graph of the corresponding classical solution has an inflection point
there.  This can be seen in the right part of the picture where the black
lines correspond to the solutions of $u'=c$ and the red and green curves to
solutions of $(u')^{2}+u^{2}+x^{2}=1$.  Obviously, the black lines are
tangent to the coloured curves at the marked intersection points. But the
red curve crosses the black line, whereas the green curve stays on one
side.
\end{example}

\begin{example}\label{ex:cone}
  To conclude this section, we study equations with ``intrinsic'' algebraic
  singularities, i.\,e.\ singularities that are not solely due to the
  intersection of irreducible components.  Some classical examples can
  already be found in the work of Ritt.  He studied for instance the
  equation $(u')^{2}-4u^{3}=0$ \cite[II.\textsection19]{ritt:da}. Here, all
  points $(x,0,0)$ are algebraic singularities, whereas all other points on
  the corresponding algebraic differential equation $\J_{1}$ are regular.
  As the differential Thomas decomposition applied in the first step of
  Algorithm~\ref{alg:regularitydecomposition} shows, a singular integral,
  namely the solution $u(x)=0$, exists here besides the generic component.
  Obviously, our algebraic singularities just form the graph of the first
  prolongation of this solution.  When we apply
  Algorithm~\ref{alg:simpleregularitydecomposition} to the generic
  component, then it uses the inequations in the entered differential
  system only for the saturation; otherwise they are ignored.  Hence the
  analysed algebraic differential equation $\J_{1}$ is the full variety
  corresponding to the given equation.  In particular, $\J_{1}$ contains
  all the algebraic singularities, but
  Algorithm~\ref{alg:simpleregularitydecomposition} recovers them and puts
  them again into a separate regularity component.  The singular integral
  represents here a kind of limit towards which all the other solutions
  tend asymptotically.

  As a second example, we consider the cone in the first-order jet bundle,
  i.\,e.\ we study the scalar differential equation $\J_{1}$ given by
  \begin{displaymath}
    (u')^2 - u^2 - x^2 = 0
  \end{displaymath}
  which obviously possesses an isolated algebraic singularity at the
  origin.  The regularity decomposition of $\J_{1}$ determined with our
  algorithm yields two components: one consisting solely of this algebraic
  singularity and one containing all other points which are regular.

  \begin{figure}
    \centering
    \includegraphics[height=5cm]{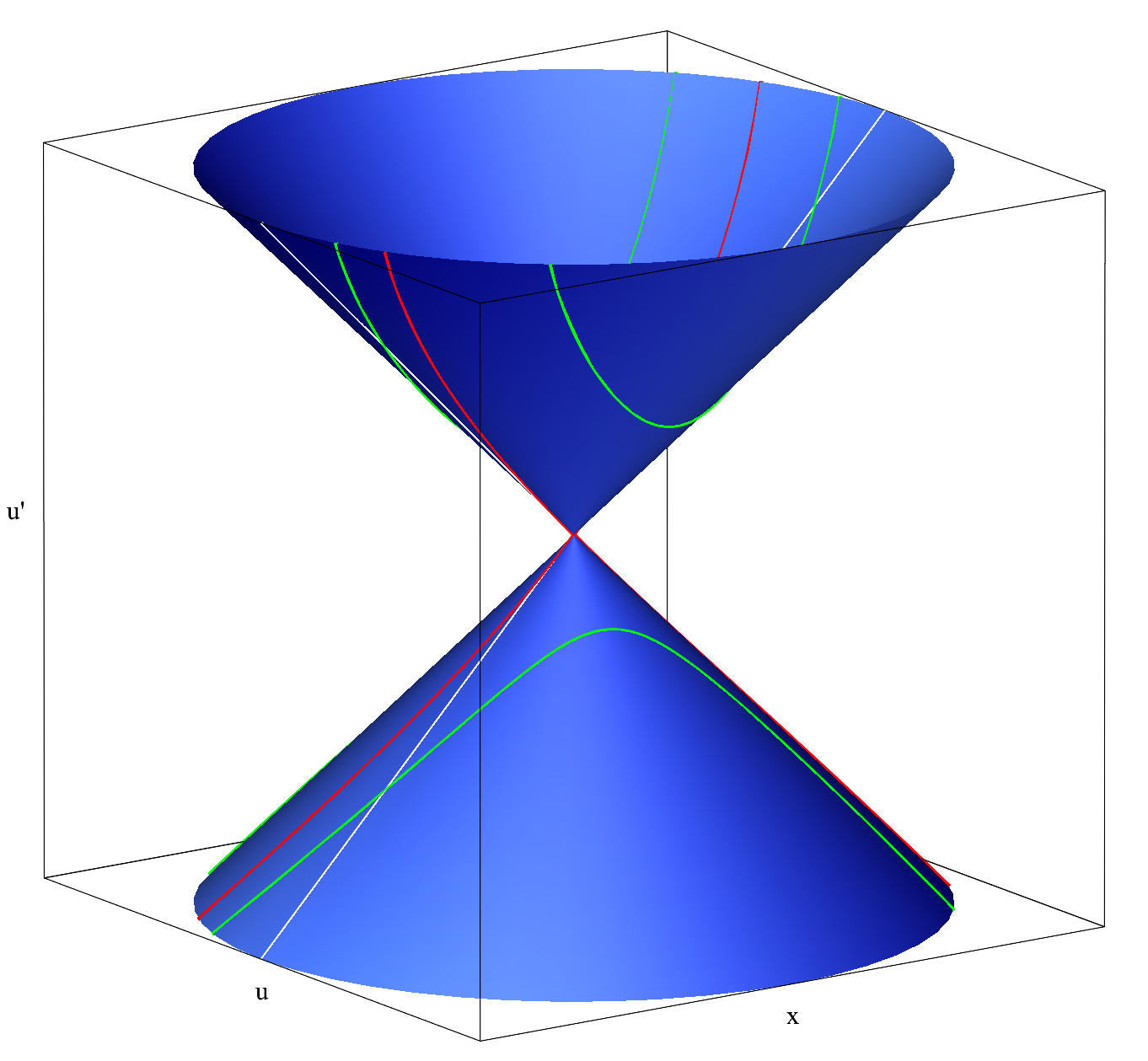}\quad
    \includegraphics[height=5cm]{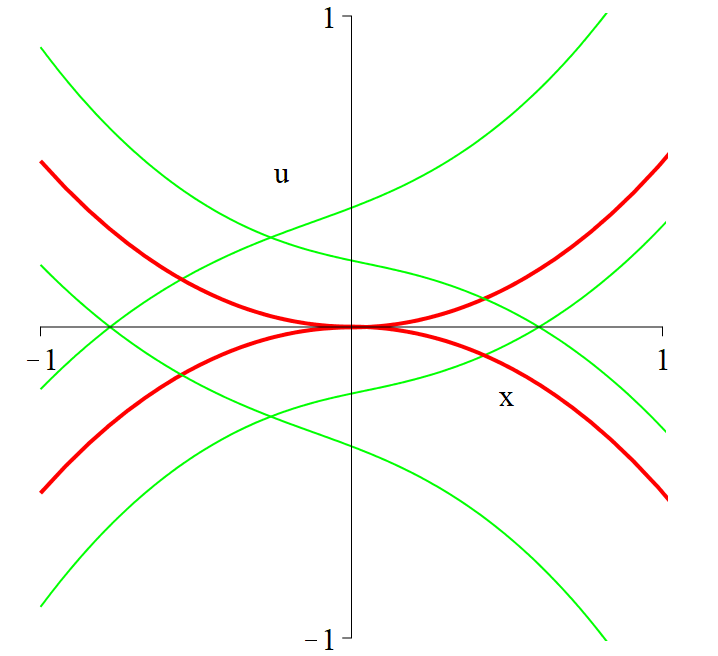}
    \caption{Generalised solutions going through an algebraic singularity
      of a real first-order differential equation. Left: situation in
      $J_{1}\pi$. Right: projection to $x$-$u$ plane.}
    \label{fig:cone}
  \end{figure}

  It is of obvious interest to study the local solution behaviour around
  this algebraic singularity and again we find a much wider range of
  possibilities over the real numbers.  In our case, a real analysis can be
  performed with a simple ad hoc approach.  Around any regular point
  $(\bar{x},\bar{u},\bar{u}')\in\J_{1}$, the Vessiot distribution is
  generated by the vector field
  $X=u'\partial_{x}+(x^{2}+u^{2})\partial_{u}+(x-uu')\partial_{u'}$. Note
  that $X$ vanishes when one approaches the origin.  By restricting to
  either the lower or the upper half cone, we can express $u'$ by $x$ and
  $u$ and project to the $x$-$u$ plane obtaining the vector field
  $Y=\pm\sqrt{x^{2}+u^{2}}\partial_{x}+(x^{2}+u^{2})\partial_{u}$.  It can
  trivially be continued to the origin where it vanishes.  However, it is
  not differentiable at this point.  Therefore, its behaviour at this
  stationary point cannot be decided using the Jacobian
  matrix. Transforming to polar coordinates (i.\,e.\ performing a blow up)
  shows that there is a unique invariant manifold going through the
  algebraic singularity which corresponds to the graph of a (prolonged)
  solution.  We obtain one such solution from each half cone (see the red
  curves in Figure~\ref{fig:cone}).  As the graphs of both solutions possess
  a horizontal tangent at the origin, it is possible to ``switch'' at the
  singularity from one to the other.  Hence, we find that our equation
  possesses exactly four $\mathcal{C}^{1}$ solutions for the initial
  condition $u(0)=0$ and $u'(0)=0$.  By analysing the prolongations of our
  equation, it is not difficult to verify that the solutions that stay
  inside of one half cone are even smooth, whereas the ``switching''
  solutions are only $\mathcal{C}^{1}$, as their second derivative jumps
  from $1$ to $-1$ or vice versa at $x=0$.  Figure~\ref{fig:cone} also shows
  in white the Vessiot cone at the algebraic singularity which consists of
  two intersecting lines.  One sees that they are indeed the tangents to
  the prolonged solutions through the singularity.
\end{example}

\section{Conclusions}\label{sec:conclusions}

We developed a framework for the detection of all singularities of an
arbitrary differential system with polynomial non-linearities at a fixed
order.  It is based on the notion of an algebraic jet set
(Definition~\ref{def:ade}) and covers both ordinary and partial differential
equations.  Our framework merges concepts from differential topology with
tools from differential algebra and algebraic geometry.  In particular for
partial differential equations, it provides the first general and rigorous
definition of singularities.  While we could not prove that the taxonomy of
Definition~\ref{def:regsing} is complete for systems which are not of
finite type, our first main result, Theorem~\ref{thm:regdense}, shows that
the definition is meaningful in the basic sense that regular points
represent the generic case.

We augmented the classical theory of singularities of differential
equations by the novel notion of a regularity decomposition
(Definition~\ref{def:regdecomp}), which is based on the concept of a regular
algebraic jet set and in particular allows for a rigorous handling of
situations where singularities are not isolated.  A regularity
decomposition essentially decomposes an algebraic jet set into subsets on
which all relevant geometric structures show a uniform behaviour.  Our
second main result, Theorem~\ref{thm:algregdecomp}, provides an algorithmic
proof for the existence of regularity decompositions for arbitrary simple
differential systems.

Finally, we solved a long standing problem in the geometric theory of
differential equations: the construction of effectively provably regular
equations.  Most results in the geometric theory assume that one is dealing
with a regular differential equation.  However, to the best of our
knowledge, nobody has so far provided an effective criterion for checking
whether or not a given differential equation is regular.  The basic problem
is that such a criterion must take into account all infinitely many
prolongations of the considered differential equation.  Our third main
result, Theorem~\ref{thm:regde}, shows that the regularity decomposition
determined by our algorithm contains in each prime component of the given
system a unique regularity component which defines a regular differential
equation.

Our approach is based on both the differential and the algebraic Thomas
decomposition and therefore fully algorithmic.  An algebraic Thomas
decomposition is crucial for the detection of all singularities.  However,
as discussed in Example~\ref{ex:hypgather}, such a decomposition yields in
general more than we really need, as it also takes into account the
geometry of the embedding of the given algebraic differential equation into
the ambient jet bundle.  From a theoretical point of view, these
unnecessary case distinctions are ugly but harmless.  From a computational
point of view, they considerably increase the computational costs and thus
it would be useful to find a way to avoid them.  Based on the existing
implementation of these decompositions in \textsc{Maple}
\cite{bglr:thomasalg} and the built-in \textsc{Maple} procedure for prime
decomposition, it is straightforward to implement our
Algorithms~\ref{alg:simpleregularitydecomposition} and
\ref{alg:regularitydecomposition} for constructing a regularity
decomposition in \textsc{Maple}.  Indeed, one of the authors (MLH) provided
such an implementation and all examples in this work have been computed
with it.

Our results lead immediately to a number of new questions.  The most
obvious one concerns the \emph{local solution behaviour} around
singularities, in particular the existence of solutions connecting two or
more regularity components.  Its investigation requires first an analysis
of the ``neighbourhood relationships'' of the found components, i.\,e.\
does a certain component lie in the Zariski closure of another component?
Such information can be straightforwardly obtained by classical Gr\"obner
bases techniques (cf.\ e.\,g.\ \cite{clo:iva}).  A deeper study of the local
solution behaviour requires additional methods which are beyond the
scope of this work.  Furthermore, such a study can most probably not be
done at the same level of generality as this work; one has to
specialise to more specific classes of systems.

For geometric singularities of ordinary differential equations considered
\emph{over the real numbers} much is already known from the works in the
context of differential topology cited in the Introduction.  Typical
questions here are existence, (non)uniqueness and regularity of one- and
two-sided solutions.  At regular singularities the situation is fairly
simple: they are generically either the initial or the terminal point of
two classical solutions (thus generically only one-sided solutions exist at
such points).  A precise formulation covering also non-generic situations
can e.\,g.\ be found in \cite[Thm.~4.1]{wms:aims}.  For the analysis of
irregular singularities, one can employ dynamical systems theory, as
usually the Vessiot distribution is generated outside of an irregular
singularity by a vector field which vanishes at the singularity.
Generalised solutions through the singularity can then be constructed as
one-dimensional invariant manifolds and typically several (possibly even
infinitely many) solutions intersect at such a singularity.\footnote{In
  low-dimensional situation, it is useful to be able to actually see the
  singularities and solutions through and around them.  In
  \cite{bss:numvis}, a \textsc{Matlab} toolbox for producing corresponding
  2D and 3D plots is presented.}  A detailed analysis of a specific class
of scalar quasilinear\footnote{It should be noted that quasilinear
  differential equations possess a special geometry, as here the Vessiot
  distribution is projectable \cite{wms:singbif} leading to phenomena not
  arising in the fully nonlinear equations usually studied in differential
  topology.  Using classical analytical techniques, such equations have
  been analysed in some detail e.\,g.\ in \cite{rr:dae}.} second-order
ordinary differential equations along these lines can be found in
\cite{wms:quasilin}.  In particular, it is shown there how regularity
questions can be answered geometrically by studying prolongations.

For linear ordinary differential equations, the analysis of singularities
\emph{over the complex numbers} has a long tradition going back at least to
the classical works of Fuchs and Frobenius which is nowadays often
considered as part of differential Galois theory (cf.\ \cite{ps:galois} and
references therein).  Note that in this context the terminology regular and
irregular singularity is often used with a different meaning than in this
work.  In a complex setting, the regularity of solutions is of course no
issue.  Instead one studies questions like the monodromy of multivalued
solutions or the Stokes phenomenon (cf.\ e.\,g.\ \cite{hz:monogr} or
\cite{ww:aeode}) which are both from a theoretical and an algorithmic point
of view still far from being solved.

We mentioned already in Remark~\ref{rem:regular} that for partial
differential equations the taxonomy of Definition~\ref{def:regsing} might
be incomplete.  The deeper problem behind this question is to define
precisely what in this case the \emph{regular behaviour} should be.  For
equations of finite type, the prolonged solutions lead to a foliation of
the differential equation around any regular point, as in this case the
vanishing of the symbol space implies that the Vessiot distribution itself
is the unique complement to the symbol space and its integral manifolds
form the leaves of a (unique) foliation by the Frobenius theorem.  If the
differential equation is not of finite type, infinitely many possible
complements exist and each of them leads to a different foliation by its
integral manifolds.  Here it is still unclear whether our definition is
already sufficient to avoid any possible kind of singular behaviour.  For
regular differential equations, the different complements can be
constructed by solving a combined algebraic-differential system which
arises out of the structure equations of the Vessiot distribution (see the
discussion in \cite{wms:vessconn2}).  It has not yet been studied how this
construction is affected by singularities and whether further kinds of
singularities may be hidden in the structure equations.

The study of solutions around \emph{algebraic singularities} has not found
much attention yet.  Within differential topology, they simply do not
occur, as it is always assumed that one is dealing with a
manifold. Recently, Falkensteiner and Sendra \cite{fs:fpsaaode} used the
classical theory of algebraic curves to study formal power series solutions
of autonomous algebraic ordinary differential equations of first order by
relating them to places.  However, an extension of their approach to higher
dimensional situations appears to be highly nontrivial.  Our analysis in
the non-autonomous Example~\ref{ex:cone} corresponding to an algebraic
surface was performed in a rather ad hoc manner, but the principal idea
should be extendable to more complicated situations, as the definition of a
simple algebraic system means that each equation in the system is solvable
for its leader.  Thus one can at least in principle obtain an explicit
expression for a vector field generating the Vessiot distribution (for
ordinary differential equations), as we used it in the example.

Our approach studies the singularities in a fixed order $\ell$.  In this
work, we have only been concerned with choosing $\ell$ sufficiently high
for a meaningful analysis.  An obvious question is how regularity
decompositions in order $\ell$ and in order $\ell+1$ are related or, more
generally, the behaviour of singularities under \emph{prolongations}.  It
is related to classical decidability questions for power series solutions
as e.\,g.\ studied by Denef and Lipshitz \cite{dl:pss}.  It is easy to see
that at a regular singularity no power series solutions can exist, as the
fibre above it is always empty.  The fibre over an irregular singularity
consists entirely of singularities, but it is not clear of which type.  A
power series solution can exist at such a point only, if at each order of
prolongation the corresponding fibre contains at least one irregular
singularity.  Thus we meet again the problem of checking infinitely many
conditions.  To the best of our knowledge, it is still unknown whether one
can decide the existence of power series solutions for given initial data
with a finite algorithm.

The algorithms behind the algebraic and the differential Thomas
decomposition require that the base field is algebraically closed.  For
this reason, we considered in this work exclusively differential equations
over the complex numbers.  From an application point of view, it is of
great interest to have a similar theory as developed in this work for
\emph{real} differential equations.  A first step in this direction can be
found in \cite{sss:realsing} for ordinary differential equations.  There
the algebraic Thomas decomposition is replaced by a parametric Gaussian
algorithm followed by a quantifier elimination.  This process represents a
suitable alternative for the effective detection of real singularities and
as a by-product avoids to some extent the above mentioned problem that the
algebraic Thomas decomposition leads to unnecessary case distinctions
because of the geometry of the embedding of the differential equation.  As
demonstrated in \cite{sss:realsing}, an analysis of
Example~\ref{ex:hypgather} leads now to no redundant cases.

\appendix

\section{Algebraic Systems and the Algebraic Thomas Decomposition}
\label{sec:algThomas_appendix}

We fix a total ordering (or \emph{ranking}) on the variables of the
polynomial ring $\P=\CC[x^{1},\dots,x^{n}]$ by setting $x^{i}<x^{j}$ for
$i<j$.  The greatest variable with respect to $<$ appearing in a
non-constant polynomial $p\in\P$ is called the \emph{leader} of $p$ and
denoted by $\ld{(p)}$; for $p\in\CC$ we set $\ld{(p)}=1$.  We regard every
polynomial $p\in\P\setminus\CC$ as a univariate polynomial in the
indeterminate $x^{k}:=\ld{(p)}$.  Then the coefficients of $p$ as a
polynomial in $x^{k}$ are contained in $\CC[x^{i} \mid 1 \le i < k]$.  The
coefficient of the highest power of $\ld{(p)}$ in $p$ is called the
\emph{initial} of $p$ and denoted by $\init{(p)}$.  Finally, we introduce
the \emph{separant} of $p$ as $\spt{(p)}:=\partial p/\partial x^{k}$.

An \emph{algebraic system} $S$ is a finite set of polynomial equations and
inequations
\begin{displaymath}
S = \bigl\{\ p_1 = 0, \, \ldots, \, p_s = 0, \, 
q_1 \neq 0, \, \ldots, \, q_t \neq 0\ \bigr\}
\tag{A}\label{eq:algebraicsystem}
\end{displaymath}
with polynomials $p_i, q_j \in\P$ and $s,t \in \NN_{0}$.  Its
\emph{solution set} is defined as
\begin{displaymath}
\Sol{(S)}:=\bigl\{\ a \in\CC^n \mid p_i(a) = 0, \, q_j(a) \neq 0 
\mbox{ for all } i, j\ \bigr\}\;.
\end{displaymath}
Obviously, $\Sol{(S)}$ is a locally Zariski closed set, namely the
difference of the two varieties $\Sol{(\{p_{1}=0,\ldots,p_{s}=0\})}$ and
$\Sol{(\{q_{1}\cdots q_{t}=0\})}$.

\begin{definition}\label{def:algsimple}
  An algebraic system $S$ as in \eqref{eq:algebraicsystem} is said to be
  \emph{simple} (with respect to the ranking $<$), if the following three
  conditions hold:
  \begin{enumerate}[(i),nosep]
  \item All equations and inequations have pairwise different leaders,
    i.\,e.\ we have
    $\bigl|\,\{\, \ld(p_1), \ldots, \ld(p_s), \ld(q_1), \ldots, \ld(q_t)\,
    \}\setminus\{1\}\,\bigr| = s+t$
    (\emph{triangularity}).\label{def:algsimple:triangularity}
  \item For every $r \in \{ p_1, \ldots, p_s, q_1, \ldots, q_t \}$, the
    equation $\init{(r)} = 0$ has no solution in $\Sol{(S)}$
    (\emph{non-vanishing initials}).\label{def:algsimple:initial}
  \item For every $r \in \{ p_1, \ldots, p_s, q_1, \ldots, q_t \}$, the
    equation $\spt{(r)} = 0$ has no solution in $\Sol{(S)}$
    (\emph{square-freeness}).\label{def:algsimple:squarefree}
  \end{enumerate}
\end{definition}

We associate with the simple algebraic system $S$ the saturated ideal
\begin{equation}\label{eq:Ialg}
\I_{\mathrm{alg}}(S):=\langle p_1,\ldots,p_s\rangle:q^\infty\subset\P
\qquad\mbox{where\ }q=\init{(p_1)}\cdots\init{(p_s)}\,.
\end{equation}
According to \cite[Prop.~2.2.7]{dr:habil}, it represents the vanishing
ideal of the Zariski closure of $\Sol(S)$, i.\,e.\ the ideal of all
polynomials in $\P$ which vanish on $\Sol(S)$. In particular,
$\I_{\mathrm{alg}}(S)$ is always a radical ideal.

Simple systems are a special class of algebraic systems for which the
solution set can be obtained iteratively by finding zeros of univariate
polynomials. First observe that triangularity implies that the simple
system $S$ contains either at most one equation $p(x^1) = 0$ with leader
$x^1$ or at most one inequation $q(x^1) \neq 0$ with leader $x^1$. The
number of zeros of $p$ (of $q$) in $\CC$ is equal to the degree of $p$ (of
$q$, respectively) due to square-freeness. In the former case, any zero
$a^1 \in \CC$ of $p$ can be chosen for the coordinate $x^1$ of a solution
of $S$.  In the latter case, all elements of $\CC$ except the zeros of $q$
can be chosen instead. If $S$ does not contain any equation or inequation
with leader $x^1$, then $a^1$ is arbitrary.  We substitute $a^1$ for $x^1$
in the equation or inequation with leader $x^{2}$ in $S$ leading to a
univariate polynomial in $x^{2}$. The degree of this polynomial is
independent of the choice of $a_1$ due to the non-vanishing initial. Again
because of square-freeness, the number of zeros of this polynomial is equal
to its degree. By iterating this process, we obtain a solution
$(a^1, a^2, \ldots, a^n) \in \CC^n$ of $S$ and every solution of $S$ can be
obtained in this way. This process makes use of the fact that the
projections from the solution set of $S$ onto the subspace with coordinates
$x^1$, $x^{2}$, \ldots, $x^k$ have uniform fibre cardinality
\cite{wp:counting}.

\begin{definition}
  Let $S$ be an algebraic system as in (\ref{eq:algebraicsystem}).  A
  \emph{Thomas decomposition} of it consists of finitely many simple
  algebraic systems ${S_1, \ldots, S_k}$ such that the solution set
  $\Sol{(S)}$ is the disjoint union of $\Sol{(S_1)}, \ldots, \Sol{(S_k)}$.
\end{definition}

Thomas \cite{th:ds,th:sr} proved that any algebraic system admits a
(non-unique) Thomas decomposition.  Using subresultants and case
distinctions, it can be algorithmically determined \cite{bglr:thomasalg}.
An implementation in \textsc{Maple} is described in \cite{bl:thomasimpl}.

\section{Differential Systems and the Differential Thomas Decomposition}
\label{sec:diffideal}

We proceed to the \emph{differential polynomial ring} (see
\cite{ritt32,ritt:da,wys:rkdp} for more information on the for us relevant
parts of differential algebra).  Let $K=\CC(x^{1},\ldots,x^{n})$ be the
field of rational functions on $\CC^n$ and $\delta_i$ the derivation
$\partial/\partial x^{i}$.  Given a set of \emph{differential
  indeterminates} $U=\{ u^1, \ldots, u^m \}$, we define the ring of
differential polynomials as the polynomial ring
$K\{U\}:=K\bigl[\,u^\alpha_\mu \mid 1 \le \alpha \le m, \, \mu \in
\NN_{0}^n\,\bigr]$ in the infinitely many variables $u^\alpha_\mu$.  The
derivations $\delta_i: K \to K$ extend to derivations
$\delta_i: K\{U\} \to K\{U\}$ via
$\delta_i(u^\alpha_\mu) := u^\alpha_{\mu+1_i}$, additivity, and the Leibniz
rule.  Here $1_i$ is the multi-index of length $n$ whose entries are $0$
except for the $i$-th entry which is $1$.  We define
$\delta^\mu := \delta_{1}^{\mu_1} \ldots \delta_{n}^{\mu_n}$ and
$|\mu| := \mu_1 + \ldots + \mu_n$, the length of any multi-index
$\mu \in \NN_{0}^n$.  Given differential polynomials
$p_{1},\ldots,p_{s}\in K\{U\}$, we distinguish between the algebraic ideal
$\langle p_{1},\ldots,p_{s}\rangle$ consisting of all linear combinations
of them and the differential ideal
$\langle p_{1},\ldots,p_{s}\rangle_{\Delta}$ containing in addition all
differential consequences $\delta^{\mu}p$ of any element $p$ of it.

We introduce the subring $\D\subset K\{U\}$ of those differential
polynomials where also the coefficients are polynomials in the variables
$x^{i}$.  Moreover, for any $\ell\in\NN_{0}$ we consider the subalgebra
\begin{displaymath}
\D_{\ell}=\CC\bigl[\,x^{i},\,u^\alpha_\mu \mid 
1\leq i\leq n, \, 1 \le \alpha \le m, \, 
|\mu|\leq\ell\,\bigr]
\end{displaymath}
which is the coordinate ring of the affine space $\AA_{\CC}^{d}$ where
$d=n+m\tbinom{n+\ell}{\ell}$.  Later, we identify the jet bundle
$\je{\ell}$ of the geometric theory (see Section~\ref{sec:geojet} below)
with the affine space $\AA_{\CC}^d$ and consider $\D_{\ell}$ as its
coordinate ring. Consequently, we call the variables $u^\alpha_\mu$ of the
polynomial ring $K\{U\}$ \emph{jet variables}.

A \emph{ranking} on the differential polynomial ring $K\{U\}$ is a total
ordering $<$ on the set of jet variables $u^\alpha_\mu$ such that
$u^\alpha < \delta_i u^\alpha$ for all $i$ and $\alpha$, and such that
$u^\alpha_\mu < u^{\alpha'}_{\mu'}$ implies
$\delta_i u^\alpha_\mu < \delta_i u^{\alpha'}_{\mu'}$ for all $i$,
$\alpha$, $\alpha'$, $\mu$, $\mu'$.  A ranking $<$ is \emph{orderly}, if
$|\mu_1|<|\mu_2|$ implies $u^{\alpha_1}_{\mu_1} < u^{\alpha_2}_{\mu_2}$ for
all $\alpha_1$, $\alpha_2$, $\mu_1$, $\mu_2$. A \emph{Riquier ranking}
satisfies the following property: if the relation
$u^\alpha_\mu<u^\alpha_{\mu'}$ holds for one value of the index $\alpha$,
then it must hold for all values of $\alpha$ (the meaning of this condition
is discussed in \cite[p.~428]{wms:invol}).  The
definitions of leader, initial and separant given above can be extended
straightforwardly.

A \emph{differential system} $S$ is given by a finite set of differential
polynomial equations and inequations
\begin{displaymath}
S = \bigl\{\, p_1 = 0, \, \ldots, \, p_s = 0, \, 
q_1 \neq 0, \, \ldots, \, q_t \neq 0 \,\bigr\}\tag{D}
\label{eq:differentialsystem}
\end{displaymath}
with $p_i,q_j \in K\{U\}$ and $s,t \in \NN_{0}$. Note that by clearing
denominators we may (and will) always assume that actually
$p_i,q_j \in \D$.

As always for differential equations, the issue arises what kind of
functions are permitted as solutions. We use here mainly local holomorphic
functions $f:\U\rightarrow\CC$ defined on some metric open domain
$\U\subseteq\CC^{n}$.  However, in our approach the actual nature of the
considered functions is not so important and we could equally well work
with formal power series or meromorphic functions.  In the sequel, we
simply assume that some set of functions admissible as solutions has been
fixed and we denote by $\Sol{(S)}$ the set of solutions in this set. We
further assume that a differential Nullstellensatz holds for this set.
This is needed to establish a one-to-one correspondence between the
solution sets of differential systems and the radical differential ideals
of the differential polynomial ring. For a system of differential equations
in $K \{ U \}$ with our choice of $K$, a differential Nullstellensatz holds
for local holomorphic functions (see e.\,g.\
\cite{raudenbush:nullstellen,ritt32}).

\begin{theorem}[Nullstellensatz for Holomorphic Functions]\label{thm:dnss}
  Let $p_1,\dots,p_s \in K\{ U \}$ be differential polynomials and
  $I=\langle p_1,\dots,p_s\rangle_{\Delta}$ the differential ideal
  generated by them.  Moreover, let $q \in K\{ U \}$ be a differential
  polynomial which vanishes for all local holomorphic solutions of
  $I$. Then some power of $q$ is an element of $I$.
\end{theorem}

The concept of \emph{passivity} introduced by Riquier \cite{riq:edp} and
Janet \cite{mj:book} represents a differential algebraic version of
completeness or formal integrability.  For lack of space, we cannot recall
here all the required definitions, but refer to \cite{vpg:decomp} for a
modern presentation of the form in which it is used here. Riquier
\cite[Chapt.~VII, \textsection115]{riq:edp} showed how one can formulate
for a passive system an initial value problem (see
\cite[Sect.~9.3]{wms:invol} for a modern formulation of this construction)
admitting an existence and uniqueness theorem for holomorphic solutions.

\begin{theorem}[Riquier's Theorem]\label{thm:riq}
  Let $<$ be an orderly Riquier ranking. Then for a system of holomorphic
  differential equations which is orthonomic and passive with respect to
  $<$ the corresponding initial value problem possesses for holomorphic
  initial data locally a unique holomorphic solution.
\end{theorem}

The assumption of passivity allows for the algorithmic construction of
formal power series solution for any ranking (see
Remark~\ref{rem:powerseries} below).  In the case of an orderly Riquier
ranking, one can then prove the convergence of this series obtaining the
above theorem. Orthonomic means that each equation can be solved in a
unique manner for its leader. Obviously, a general implicit differential
equation does not satisfy this condition.  For this reason, we need as in
the algebraic case the notion of a simple system permitting us the use of
Riquier's Theorem.

\begin{definition}\label{def:diffsimple}
  The differential system $S$ given by (\ref{eq:differentialsystem}) is
  \emph{simple} (with respect to a given ranking $<$), if the following
  conditions hold:
  \begin{enumerate}[(i),nosep]
  \item $S$ is simple as an algebraic system (in the finitely many jet
    variables $u^\alpha_\mu$ which actually occur in $S$ ordered according
    to $<$).\footnote{We consider here the independent variables $x^{i}$ as
      part of the coefficient field.  One should also note that if only the
      inclusion of these variables yielded an algebraically simple system,
      then $S$ would be differentially inconsistent.}
  \item $\{p_1,\dots,p_s\}$ is a passive system (for the Janet division).
  \item No leader of an inequation $q_j$ is an (iterated) derivative of the
    leader of an equation $p_k$.
  \end{enumerate}
\end{definition}

\begin{definition}
  A \emph{Thomas decomposition} of a differential system $S$ consists of
  finitely many simple differential systems $S_1, \ldots, S_k$ such that
  $\Sol{(S)}$ is the disjoint union of the solution sets
  $\Sol{(S_1)}, \ldots, \Sol{(S_k)}$.
\end{definition}

Thomas \cite{th:ds,th:sr} proved also in the differential case the
existence of such decompositions.  Again, it is possible to construct them
algorithmically by interweaving algebraic Thomas decompositions and the
Janet-Riquier theory \cite{bglr:thomasalg}.  The resulting algorithm is
implemented in \textsc{Maple} \cite{bl:thomasimpl,glhr:tdds}.

\begin{remark}\label{rem:powerseries}
  For a \emph{simple} differential system $S$ it is possible to construct
  systematically formal power series solutions.  Let $\ell$ be the maximal
  order of an equation or an inequation in $S$ and add to $S$ all partial
  derivatives of order at most $\ell$ of the equations in $S$.  Now we
  choose an expansion point $x_0=(x^1_0,\ldots,x^n_0)\in\CC^n$ such that
  all equations and inequations in $S$ are defined at $x = x_0$ and no
  initial and no separant vanishes for $x = x_0$. Hence, a formal power
  series solution is of the form
  $u^\alpha=\sum_{\mu\in\NN_{0}^n}c_\mu^\alpha\frac{(x-x_0)^\mu}{\mu!}$.
  We choose $c^{\alpha}_{\mu}\in\CC$ for all derivatives $u^{\alpha}_{\mu}$
  up to order $\ell$.  These choices must be performed in such a manner
  that after substituting $x$ by $x_{0}$ and all $u^{\alpha}_{\mu}$ by the
  corresponding constants $c^{\alpha}_{\mu}$ no initial or separant of an
  equation or inequation vanishes and all equations and inequations are
  satisfied.  If $u^{\alpha}_{\mu}$ is the leader of an equation, then only
  finitely many values are possible for $c^{\alpha}_{\mu}$.  If it is the
  leader of the derivative of an equation, then there is no freedom in
  choosing $c^{\alpha}_{\mu}$, as any differentiated equation is linear in
  its leader.  If $u^{\alpha}_{\mu}$ is the leader of an inequation, then
  all but finitely many values are possible for $c^{\alpha}_{\mu}$.  For
  all other jet variables $u^{\alpha}_{\mu}$ up to order $\ell$, the
  constants $c^{\alpha}_{\mu}$ can be chosen completely freely.
	
  The jet variables $u^{\alpha}_{\mu}$ of an order greater than $\ell$ can
  be partitioned into two disjoint sets.  For those which are not the
  derivative of the leader of an equation in $S$, the corresponding
  constant $c^{\alpha}_{\mu}$ can be chosen arbitrarily.  For all remaining
  ones the constants $c^{\alpha}_{\mu}$ are uniquely determined by some
  derived equations, which are quasi-linear. The properties of a simple
  differential system (in particular, the passivity) ensure that now the
  formal power series
  $u^\alpha=\sum_{\mu\in\NN_{0}^n}c_\mu^\alpha\frac{(x-x_0)^\mu}{\mu!}$
  with $1\leq\alpha\leq m$ define a solution of $S$ around $x_{0}$.
\end{remark}

\section{The Geometry of Differential Equations}
\label{sec:geojet}

Since the algebraic tools used in the algorithms developed in this work
require an algebraically closed field, we concentrate on complex
differential equations.  Thus in the sequel all manifolds\footnote{For us,
  manifolds have the same local dimension at every point and thus look
  locally like an open subset of some $\CC^{d}$ with a fixed $d$.} are
complex and all variables are to be understood as complex-valued.
Restricting to holomorphic sections, one can define jet bundles in the
familiar way and there are no changes with respect to the real theory
outlined in standard references like
\cite{kl:geode,pom:eins,sau:jet,wms:invol}.

The basic geometric setting is a fibred manifold $\pi:\E\rightarrow\X$
(i.\,e.\ $\pi$ is a surjective submersion).  The coordinates on the
base space $\X$ are the independent variables $x^{1},\dots,x^{n}$; the
fibre coordinates $u^{1},\dots,u^{m}$ represent the dependent
variables or unknown functions.  The $\ell$th order \emph{jet bundle}
$\je{\ell}$ consists of all Taylor polynomials of degree $\ell$.
Naturally induced coordinates on it are thus in addition all
derivatives of the $u^i$ up to order $\ell$; we use for them the usual
multi-index notation $u^{\alpha}_{\mu}$ where $\mu\in\Nn$ is a
multi-index of length $n$.  In the sequel, these natural coordinates are
called \emph{jet variables}.  For convenience, we identify
$\E=\je{0}$.

Functions are replaced in the geometric framework by (local)
sections:\footnote{For notational simplicity, we almost always omit the
  domain of definition $\U$ and use a seemingly global notation.  However,
  all statements in this work are of a local nature.} maps
$\sigma:\U\subseteq\X\rightarrow\E$ such that $\pi\circ\sigma=\id{\U}$.
Locally, any section can be written in the form
$\sigma(x)=\bigl(x,s(x)\bigr)$ with a local holomorphic function
$s:\CC^{n}\rightarrow\CC^{m}$.  Given a section $\sigma:\X\rightarrow\E$,
\emph{prolongation} yields a section of the jet bundle
$j_{\ell}\sigma:\X\rightarrow\je{\ell}$ which is locally defined by
$j_{\ell}\sigma(x)=\bigl(x,s(x),s_{x}(x),\dots,s_{x\cdots x}(x)\bigr)$,
thus we simply add all partial derivatives of the function $s$ up to order
$\ell$.

The jet bundles of different orders form a natural hierarchy of fibrations
via the canonical projections
$\pi^{\ell+k}_{\ell}:\je{\ell+k}\rightarrow\je{\ell}$ which ``forget'' the
higher order Taylor coefficients.  Of particular interest are the
projections $\pi^{\ell}_{\ell-1}$ by just one order, as $\je{\ell}$ is an
affine bundle over $\je{\ell-1}$ modelled on the vector bundle
$S_{\ell}(T^{*}\X)\otimes V\pi$ \cite[Prop.~2.2.6]{wms:invol}.  The
\emph{fundamental identification} provides an isomorphism between this
vector bundle and the vertical bundle
$V\pi^{\ell}_{\ell-1}=\ker{T\pi^{\ell}_{\ell-1}}$.  In addition, every jet
bundle is fibred over the base space by the canonical projection
$\pi^{\ell}:\je{\ell}\rightarrow\X$ mapping each Taylor polynomial to its
expansion point.  This last projection is very important in our context:
whenever we speak without further details of a transversal or a vertical
vector field, it refers to this fibration $\pi^{\ell}$.

A crucial geometric structure on the jet bundle $\je{\ell}$ is the
\emph{contact distribution} $\C_{\ell}\subset T(\je{\ell})$.  In local
jet coordinates, it is generated by the following vector fields:
\begin{equation}\label{eq:confields}
\begin{aligned}
C_{i}^{(\ell)} &= \partial_{x^{i}}+ 
\sum_{\alpha}u^{\alpha}_{i}\partial_{u^{\alpha}}+
\sum_{0<|\mu|<\ell}\sum_{\alpha}u^{\alpha}_{\mu+1_{i}}
\partial_{u^{\alpha}_{\mu}}
\qquad (1\leq i\leq n)\;,\\
C^{\mu}_{\alpha} &= \partial_{u^{\alpha}_{\mu}}\qquad\qquad 
(|\mu|=\ell, 1\leq\alpha\leq m)\;.
\end{aligned}
\end{equation}
The first $n$ fields are transversal to the fibration $\pi^{\ell}$ and
encode the chain rule, whereas the remaining fields span the vertical
bundle $V\pi^{\ell}_{\ell-1}$.  Intuitively, the contact distribution
encodes the different roles played by the three different kinds of
coordinates: independent variables, dependent variables, and derivatives.
One way to express this intuition formally is given by the following
well-known result.

\begin{proposition}\label{prop:contact}
  A section $\gamma:\X\rightarrow\je{\ell}$ of the $\ell$th jet bundle is a
  prolongation, i.\,e.\ of the form $\gamma=j_{\ell}\sigma$ for a section
  $\sigma:\X\rightarrow\E$, if and only if
  $T(\im{\gamma})\subseteq\C_{\ell}$.
\end{proposition}

The following intrinsic definition of a differential equation does not
distinguish between scalar equations and systems.  It automatically
excludes the appearance of singularities as studied in this work.

\begin{definition}\label{def:gde}
  A \emph{differential equation} of order $\ell$ is a fibred submanifold
  $\Jc{\ell}\subseteq\je{\ell}$ such that the restriction of the canonical
  projection $\pi^{\ell}:\je{\ell}\rightarrow\X$ to the set $\Jc{\ell}$ is
  a surjective submersion.
\end{definition}

Also the notion of a solution can be easily expressed in an intrinsic
manner.  Note that the above definition of a differential equation does not
yet entail the existence of solutions, as it does not exclude hidden
integrability conditions which may lead to an inconsistency.

\begin{definition}\label{def:sol}
  A \emph{(classical) solution} of the differential equation
  $\Jc{\ell}\subseteq\je{\ell}$ is a section $\sigma:\X\rightarrow\E$ such
  that its prolongation satisfies $\im{j_{\ell}\sigma}\subseteq\Jc{\ell}$.
\end{definition}

Let $\sigma:\X\rightarrow\E$ be a classical solution of the differential
equation $\Jc{\ell}\subseteq\je{\ell}$.  Then, by definition,
$\im{j_{\ell}\sigma}\subseteq\Jc{\ell}$ is a smooth submanifold.  Hence, we
find at any point $\rho\in\im{j_{\ell}\sigma}$ that
$T_{\rho}(\im{j_{\ell}\sigma})\subseteq T_{\rho}\Jc{\ell}$.  Furthermore,
for any prolonged section
$T_{\rho}(\im{j_{\ell}\sigma})\subseteq\C_{\ell}|_{\rho}$ by
Proposition~\ref{prop:contact}.  Thus the tangential part of the contact
distribution restricted to $\Jc{\ell}$ may be considered as the space of
all infinitesimal solutions (or integral elements).

\begin{definition}\label{def:vessp}
  The \emph{Vessiot space}\footnote{In particular in the Russian
    literature, the terminology \emph{Cartan space} is more common.  We
    follow here the argumentation of Fackerell \cite{edf:vessiot} that
    Vessiot put a much stronger emphasis on the vector field side whereas
    Cartan prefered to work with differential forms.}
  $\V_{\rho}[\Jc{\ell}]$ of the differential equation
  $\Jc{\ell}\subseteq\je{\ell}$ at a point $\rho\in\Jc{\ell}$ is the set
  $\V_{\rho}[\Jc{\ell}]=T_{\rho}\Jc{\ell}\cap\C_{\ell}|_{\rho}$.  The
  family of all Vessiot spaces is briefly denoted by $\V[\Jc{\ell}]$.
\end{definition}

One should note that generally $\V[\Jc{\ell}]$ does \emph{not} define a
smooth regular distribution, as the dimension of the Vessiot spaces may
differ at different points on $\Jc{\ell}$.  It is a standard assumption in
the geometric theory (related to the notion of a \emph{regular}
differential equation) that this should not happen.

The fibration $\pi^{\ell}_{\ell-1}:\je{\ell}\rightarrow\je{\ell-1}$ allows
us to define at any point $\rho\in\je{\ell}$ the vertical space
$V_{\rho}\pi^{\ell}_{\ell-1}=\ker{T_{\rho}\pi^{\ell}_{\ell-1}}$.  We call
the vertical part of the Vessiot space at a point $\rho\in\Jc{\ell}$ the
\emph{symbol space}
$\Nc{\ell}{\rho}=\V_{\rho}[\Jc{\ell}]\cap V_{\rho}\pi^{\ell}_{\ell-1}$.  It
is not difficult to show that the Vessiot space can be decomposed as a
direct sum of linear subspaces,
$\V_{\rho}[\Jc{\ell}]=\Nc{\ell}{\rho}\oplus\H_{\rho}$, with some
$\pi^\ell$-transversal complement $\H_{\rho}$ which is not uniquely
determined.

The relationship between solutions and the Vessiot distribution is recalled
in the following well-known assertion (see e.\,g.\
\cite[Prop.~9.5.7]{wms:invol}).  One may say that the basic idea of
Vessiot's approach to differential equations consists of studying certain
subdistributions of the Vessiot distribution---which can to a large extent
be analysed by elementary linear algebra---instead of solutions themselves
(in \cite{wms:vessconn2} these subdistributions are called \emph{Vessiot
  connections}).

\begin{proposition}\label{prop:sol}
  Let the section $\sigma:\X\rightarrow\E$ be a solution of the
  differential equation $\Jc{\ell}\subseteq\je{\ell}$.  Then
  $T(\im{j_{\ell}\sigma})$ is an $n$-dimensional, $\pi^\ell$-transversal,
  involutive, smooth subdistribution of
  $\V[\Jc{\ell}]|_{\im{j_{\ell}\sigma}}$.  Conversely, let
  $\H\subseteq\V[\Jc{\ell}]$ be an $n$-dimensional, transversal,
  involutive, smooth subdistribution defined on some open subset of
  $\Jc{\ell}$.  Then any $n$-dimensional integral manifold of $\H$ (and
  such manifolds always exist by the Frobenius Theorem
  \cite[Thm.~C.3.3]{wms:invol}) is locally of the form
  $\im{j_{\ell}\sigma}$ for a solution $\sigma$ of $\Jc{\ell}$.
\end{proposition}

Given a smooth function $\Phi:\je{\ell}\rightarrow\CC$, its \emph{formal
  derivative} with respect to the independent variable $x^{i}$ yields a
function $D_{i}\Phi:\je{\ell+1}\rightarrow\CC$ which can be conveniently
defined via the contact fields \eqref{eq:confields}:
\begin{equation}\label{eq:fderiv}
D_{i}\Phi=C_{i}^{(\ell)}(\Phi)+
\sum_{|\mu|=\ell}\sum_{\alpha=1}^{m}
C^{\mu}_{\alpha}(\Phi)u^{\alpha}_{\mu+1_{i}}
\end{equation}
where $\mu+1_{i}$ denotes the multi-index obtained by raising the $i$th
entry of $\mu$ by one.

Assume that $\Phi$ depends on some jet variables other than only the
independent variables $x^{i}$ and that $k\geq0$ is the maximal order of
these jet variables.  Then $D_{i}\Phi$ depends on jet variables up to order
$k+1$ and is always linear in those of the maximal order (and thus
quasi-linear).  Let $\P=\CC[\xi^{1},\dots,\xi^{n}]$ be a polynomial ring in
$n=\dim{\X}$ variables and $\rho\in\je{\ell}$ an arbitrary point.  We
define the \emph{principal part} of $\Phi$ at the point $\rho$ as the
polynomial vector
\begin{equation}\label{eq:ppart}
\pp_{\rho}{\Phi}=
\sum_{|\mu|=k}\sum_{\alpha=1}^{m}
\frac{\partial\Phi}{\partial u^\alpha_\mu}(\rho)\,
\xi^{\mu}\ev_{\alpha}\in\P^{m}
\end{equation}
where $\ev_{\alpha}$ denotes the standard basis vectors in the free module
$\P^{m}$ over the polynomial ring $\P=\CC[x^{1},\dots,x^{n}]$ whose rank is
the fibre dimension $m$ of $\E$.  Note that the entries of
$\pp_{\rho}{\Phi}$ are homogeneous polynomials of degree $k$ .

Locally, the differential equation $\Jc{\ell}$ may be considered as the
zero set of some functions $\Phi_{i}:\je{\ell}\rightarrow\CC$.  We choose a
point $\rho\in\Jc{\ell}$ and let $\ell_i\leq\ell$ be the maximal order of
jet variables effectively appearing in $\Phi_{i}$ and
$\fv_{i}=\pp_{\rho}{\Phi_{i}}\in\P^{m}$ its principal part at $\rho$. The
\emph{(reduced) principal symbol module} at the point $\rho$ is now the
$\P$-module $\M[\rho]=\langle\fv_{1},\dots,\fv_{s}\rangle$ spanned by all
the principal parts.  The degree $\ell$ component of this module can be
identified with the annihilator of the symbol space $\Nc{\ell}{\rho}$ (see
\cite[Rem.~7.1.18]{wms:invol}).

\bibliographystyle{amsplain}
\bibliography{GSADE}

\end{document}